\documentclass[11pt]{amsart}
\usepackage{amsmath,amsthm,amssymb,latexsym,graphicx,verbatim,enumerate,
ifpdf,mdwlist, xspace}
\usepackage{mathabx}
\ifpdf
\usepackage{hyperref}
\else
\usepackage[hypertex]{hyperref}
\fi

\usepackage{amssymb,amsmath,amsthm}
\usepackage{graphicx}
\usepackage{cancel}
\usepackage{float}

\usepackage{color}

\newcommand{\bd}{\mathbf{d}}
\newcommand{\cA}{\mathcal{A}}

\newcommand{\cM}{\mathcal{M}}

\providecommand{\C}{\mathcal{C}}

\providecommand{\set}[1]{\lbrace #1 \rbrace}

\providecommand{\dom}{{\rm{dom}}}

\newtheorem{question}{Question}
\newtheorem{thm}{Theorem}[section]

\newtheorem{remark}[thm]{Remark}

\newtheorem{lemma}[thm]{Lemma}

\newtheorem{prop}[thm]{Proposition}
\newtheorem{corollary}[thm]{Corollary}
\newtheorem{claim}[thm]{Claim}

\theoremstyle{definition}
\newtheorem{defn}[thm]{Definition}

\theoremstyle{defn}

\renewcommand{\phi}{\varphi}

\DeclareMathOperator{\range}{range}

\newcommand{\Text}{\textnormal{\textbf{Txt}}}
\newcommand{\Inf}{\textnormal{\textbf{Inf}}}

\newcommand{\Ex}{\textnormal{\textbf{Ex}}}

\newcommand{\cont}{\text{content}}

\newcommand{\skipmm}[1]{\hspace{-#1mm}}
\newcommand{\infi}{\mathrm{inf}}

\newcommand{\Last}{\mathrm{Last}}
\newcommand{\First}{\mathrm{First}}
\newcommand{\Succ}{\mathrm{Succ}}
\newcommand{\Block}{\mathrm{Block}}

\title{Learning families of algebraic structures from informant}

\author[Bazhenov]{Nikolay Bazhenov}
\address{Sobolev Institute of Mathematics, pr. Akad. Koptyuga 4, Novosibirsk,
630090 Russia;\
Novosibirsk State University, ul. Pirogova 2, Novosibirsk, 630090 Russia
}
\email{bazhenov@math.nsc.ru}
\urladdr{bazhenov.droppages.com}

\author[Fokina]{Ekaterina Fokina}
\address{Institute of Discrete Mathematics and Geometry, Vienna University of Technology, Austria}
\email{ekaterina.fokina@tuwien.ac.at}
\urladdr{dmg.tuwien.ac.at/fokina}

\author[San Mauro]{Luca San Mauro}
\address{Institute of Discrete Mathematics and Geometry, Vienna University of Technology, Austria}
\email{luca.san.mauro@tuwien.ac.at}
\urladdr{dmg.tuwien.ac.at/sanmauro}

	
\thanks{Bazhenov was supported by the Russian Science Foundation, project No.~18-11-00028. 
San Mauro was supported
by the Austrian Science Fund FWF, project M 2461. \\
The authors wish to thank two anonymous referees for valuable
comments.}


\begin{document}

\maketitle

\begin{abstract}
We  combine computable structure theory and algorithmic learning theory  to study learning of families of algebraic structures. Our main result is a model-theoretic characterization of the learning type $\Inf\Ex_{\cong}$, consisting of the structures whose isomorphism types can be learned in the limit. We show that a family of structures is $\Inf\Ex_{\cong}$-learnable if and only if the structures can be distinguished in terms of their $\Sigma^{\inf}_2$-theories.  We apply this characterization to familiar cases and we show the following: there is an infinite learnable family of distributive lattices; no pair of Boolean algebras is learnable;  no infinite family of linear orders is learnable.
\end{abstract}

\section{Introduction}

In this paper we combine computable structure theory and algorithmic learning theory to study the question of extracting semantic knowledge from a finite amount of structured data.

Computable structures can be regarded as structures output by a Turing machine (with no input) step by step, where the number of steps is potentially infinite (but at most countable). At each step we observe larger and larger finite pieces of the structure: as soon as the algorithm outputs an element, it also reveals the  relations between this element and all the elements that appeared at previous stages. The algorithm can never change its mind whether a relation holds on particular elements or not. We refer the reader to Section \ref{sec:preliminaries} for a formal definition.

Looking at computable structures as described above is well-suited for an application
in inductive inference as initiated by Gold~\cite{Gold67}. Here a learner receives step by step more and more data (finite amount at each step) on an object
to be learned, and outputs a sequence of hypotheses that converges to
a finite description of the target object. In general, learning can be viewed as a dialogue between a teacher and a learner, where the learner must succeed in learning, provided the teacher satisfies a certain protocol. The formalization of this idea has two aspects: convergence behavior and teacher constraints. Again, formal definitions follow below.

Most work in inductive inference concerns either learning of formal languages or learning of general recursive functions \cite{Osh-Sto-Wei:b:86:stl, zz-tcs-08, lange2008learning}. The case of learning other structures has first been considered by Glymour~\cite{Gly:j:85} and is surveyed by Martin and Osherson~\cite{Mar-Osh:b:98}. More recently, in \cite{HaSt07, MS04,SV01} Stephan and co-authors considered learnable
ideals of rings, subgroups and submonoids of groups, subspaces of
vector spaces and isolated branches on uniformly computable sequences of trees. They showed that different types of learnability of
various classes of computable or computably enumerable structures have
strong connections to their algebraic characterizations (see, e.g., \cite[Theorem 3.1]{HaSt07}).
The fact of such correspondence between learnability from different types of information and algebraic properties of structures is of big interest from a mathematical point~of~view. 
In a sense, it is a way to study the interplay between algorithmic and algebraic properties of structures.

In this paper, we employ an approach that can be applied to an arbitrary class of computable structures. The main idea  is the following. Suppose we have a class of computable structures. And suppose we step by step get finite amounts of data about one of them. Then we learn the class, if after finitely many steps we correctly identify the structure we are observing. This is why, in this setting, we consider learning of a class of computable structures as a task of extracting semantic knowledge from finite amount of data.

In a recent paper \cite{FKS-ta} Fokina, K\"otzing and San Mauro considered learnable classes of equivalence structures. They reworked and extended the results, which appeared in Glymour~\cite{Gly:j:85}. In this paper we continue this line of investigation by applying the setup to other classes of structures. Our results (see Theorem \ref{thm:Sigma_2-theories}) are similar to Martin and Osherson's approach \cite{Mar-Osh:b:98}, but by using Turing computable embeddings, we can extract more information: in particular, we offer an upper bound to the computational power needed to learn a given family of structures (see Corollary \ref{coroll:complexity}). 

The paper is organized as follows. In Section \ref{sec:preliminaries} we give all the necessary definitions and useful facts from computable structure theory and learning theory. In Section \ref{sec:main_characterization} we prove our main result: a model-theoretic characterization of learnable families of structures. In Section \ref{sec:applications} we apply the characterization from the previous section to get examples of learnable and non-learnable classes of natural computable structures.


\section{Preliminaries}\label{sec:preliminaries}
In this section we review the necessary definitions about computable structures (Section \ref{subsect:comp_str}), infinitary formulas (Section \ref{subsect:infinitary}), and locking sequences (Section \ref{subsect:locking}).  In Section \ref{subsect:discussion}, we offer a gentle exposition to our learning paradigm, which is formally defined in Section \ref{subsect:formal}.

\smallskip

Our computability theoretic terminology is standard  and as in \cite{soare2016turing}. In particular, we denote by $\{\phi_e\}_{e\in\omega}$ a uniformly computable list of all computable functions, and by $\{\Phi^X_e\}_{e\in\omega}$ a uniformly computable list of all Turing operators with oracle $X$.

\subsection{Computable structures}\label{subsect:comp_str}

A signature is a collection of function symbols and
relation symbols that characterize an algebraic structure; a signature with no function symbol is \emph{relational}. An $L$-structure  $\mathcal{M}$  consists of a domain $M$ with an interpretation of the symbols of $L$: it is common to denote the interpretation of a function $f$ (resp.\ a relation $R$) in an $L$-structure as $f^{\mathcal{M}}$ ($R^{\mathcal{M}}$). Two $L$-structures $\mathcal{M,N}$ are isomorphic if there is a bijection $F\colon\dom(\mathcal{M})\rightarrow\dom(\mathcal{N})$ such that: 
\begin{itemize}
\item For every function symbol $g$ in $L$ of arity $n$, for all $a_1, \ldots, a_n$ in $\dom(\mathcal{M})^{n}$, $F(g^{\mathcal{M}}(a_1,\ldots,a_n))=g^{\mathcal{N}}(F(a_1),\ldots,F(a_n))$.
\item For every relation symbol $R$ in $L$ of some arity $m$, for all $a_1, \ldots, a_m$ in $\dom(\mathcal{M})$, $R^{\mathcal{M}}(a_1,\ldots,a_m)$ if and only if $ R^{\mathcal{N}}(F(a_1),\ldots,F(a_m))$.
\end{itemize}
We write $\mathcal{M}\cong\mathcal{N}$ to denote that $\mathcal{M}$ and $\mathcal{N}$ are isomorphic. 
The isomorphism is an equivalence relation on $L$-structures. The equivalence classes 
with respect to the relation $\cong$ are called \emph{isomorphism types}. We denote the isomorphism type of a structure $\mathcal{M}$ (i.e., the family of structures isomorphic to $\mathcal{M}$) as $[\mathcal{M}]_{\cong}$.

In the paper, we consider only finite signatures. When we talk about learnable families of $L$-structures, we assume that the domain of any countably infinite structure is equal to the set $\omega$ of the natural numbers. This allows us to effectively identify, through a fixed G\"odel numbering,  any sentence about such an ${L}$-structure with a natural number. We can then define the \emph{atomic diagram} $D(\mathcal{M})$ of such an ${L}$-structure $\mathcal{M}$ to be the set of $n\in\omega$ such that $n$ represents an atomic ${L}_M$-sentence true in $\mathcal{M}$ or the negation of an atomic  ${L}_M$-sentence that is false in $\mathcal{M}$. To measure the complexity of a structure, we identify it to its atomic diagram: we say that a structure $\mathcal{M}$ is $\bd$\emph{-computable} if $D(\mathcal{M})$ is a $\bd$-computable subset of $\omega$, where $\bd$ is a Turing degree. A \emph{presentation} of a 
countable algebraic structure is an arbitrary isomorphic copy $\cM'\cong \cM$ with the universe 
 a subset of $\omega$. We call a structure $\cM$ \emph{computably presentable} if it has a presentation $\cM'$ which is computable. A structure is called $\bd$-computably presentable if for some $\bd_0\leq \bd$ there exists a presentation $\cM'\cong\cM$ which is $\bd_0$-computable.

 Any computable structure $\cA$ in a  relational signature can
be presented as an increasing union of its finite substructures
\[
\cA^0\subseteq\cA^1\subseteq\ldots\subseteq\cA^i\subseteq\ldots,
\]
where $\cA^n$ is the restriction of $\cA$ to the domain $\{0,1,\ldots,n\}$  and $\cA=\bigcup_i\cA^i$.

By $\mathbb{K}_L$ we denote the class of all $L$-structures with domain $\omega$. Since the goal of our learning paradigm, as described below, is to identify the isomorphism type of structure from any of its presentations, we assume that every considered class of $L$-structures is closed under isomorphisms (modulo the restriction of the domain).

\smallskip

For additional background on computable structures, the reader is referred to \cite{AK00}.

\subsection{Informal discussion of our learning paradigm} \label{subsect:discussion}

Fokina, K{\"o}tzing, and San Mauro~\cite{FKS-ta} introduced the paradigm of informant learning for families of computably presentable structures. Before delving into the formal details, we illustrate the paradigm by considering two simple learning problems, by which we specify the following six items that characterize our paradigm: the learning domain, the hypothesis space, the information source, the prior knowledge, the criterion of success, and the learner.   The first problem, denoted as $\mathcal{P}_1$, consists in learning the family $\mathfrak{C}$, which consists of two countably infinite, undirected graphs:
\begin{enumerate}
	\item $G_1$ which contains only cycles of size two, and 
	\item $G_2$ containing only $3$-cycles.
\end{enumerate}
 
\subsubsection*{The learner} The learner is always assumed to be an algorithm.

\subsubsection*{The learning domain} Our paradigm aims at capturing the ability, or lack thereof, of learning a given structure independently of the way in which such a structure is presented. This approach is analogous with the idea, common in computable structure theory, of characterizing the sets $X$ that are \emph{coded} in a structure $\mathcal{S}$ as the sets that can be computed from any presentation of $\mathcal{S}$. Hence, the learning domain of $\mathcal{P}_1$  consists of the family $\mathfrak{C}^*$ of all possible presentations of $G_1$ and $G_2$, i.e.,  $\mathfrak{C}^*=\{H : H\cong G_1 \mbox{ or } H\cong G_2 \}$. Observe that  $\mathfrak{C}^*$ coincides with the union of the isomorphism types of $G_1$ and $G_2$. 

\subsubsection*{The hypothesis space} The hypothesis space of $\mathcal{P}_1$ is the set $\{1,2, ?\}$, where the symbols ``$1$'' and ``$2$'' means that the learner conjectures that the target graph is isomorphic, respectively,  to $G_1$ and $G_2$, and the symbol ``$?$'' means that the learner has no clue about the isomorphism type of the target graph. Notice that, since our paradigm deals with learning \emph{up to isomorphism}, it is sufficient to specify the symbols to refer to the nonisomorphic structures in $\mathfrak{C}$ (i.e., $G_1$ and $G_2$) and there is no need to extend the hypothesis space with other symbols for denoting all structures of $\mathfrak{C}^*$.

\subsubsection*{The information source}  
An informant $I$ for a graph $H$ in $\mathfrak{C}^*$ is  an infinite list of pairs containing: all
pairs $(x, y)$ of natural numbers, as the first component; and either $0$ or $1$, as the second
component, where this second component is $1$ if and only if $x$ and $y$ are adjacent in $H$.  So, each entry provided by $I$ can be regarded as a triple $(x,y,z)\in \omega\times\omega\times \{0,1\}$. We assume that, at any stage $s$, the learner receives the first $s$
triples of the informant $I$. 

This style of learning in which the learner receives both positive and negative information about the target object is called, after Gold~\cite{Gold67}, $\Inf$-learning. Learning without negative information is called
learning from text (as opposed to learning from informant) and is denoted by $\mathbf{Txt}$ instead
of $\mathbf{Inf}$. In \cite{FKS-ta}, the authors considered $\mathbf{Txt}$-learning of equivalence structures. In the present paper we focus only on 
$\Inf$-learning, postponing a systematic analysis of learning algebraic structures from text to a future work.

\subsubsection*{The prior knowledge} 
The
prior knowledge of $\mathcal{P}_1$ consists of the knowledge that the target graph is isomorphic to either
$G_1$ or $G_2$.

\subsubsection*{The criterion of success} Finally, the learning problem $\mathcal{P}_1$ is positively solved, if there is a learner that, receiving larger and larger pieces of any graph $G$ in $\mathfrak{C}^*$, eventually stabilize to a correct guess about  whether $G$ is isomorphic to $G_1$ or $G_2$. 

So, our learning paradigm is an instance of \emph{limit learning}: we allow the learner to have  an arbitrary (but finite) number of mind changes before
stabilizing on a correct conjecture. This style of learning, which dates back to Gold~\cite{Gold67}, is often called \emph{explanatory learning} (e.g., in \cite{case1983comparison}) and denoted as $\Ex$.

\medskip

Having informally specified the key items of our learning paradigm, one can easily design an algorithm for learning the family $\mathfrak{C}$:
\begin{itemize}
\item Given a graph $H$ as input, we search for a cycle of size $n\in \{2, 3\}$ inside $H$. If $n=2$, then $A_{\mathfrak{C}}$ conjectures that $H$ is a copy of $G_1$. If $n=3$, then $A_{\mathfrak{C}}$ thinks that $H\cong G_2$.
\end{itemize}

More formally, the algorithm $A_{\mathfrak{C}}$ is arranged as follows:
\begin{itemize}	
	\item We define $A_{\mathfrak{C}}(I[0]) :=\ ?$. At a stage $s+1$, proceed as follows: 
\begin{itemize}	
\item	If $A_{\mathfrak{C}}(I[s])\neq\ ?$, then just set $A_{\mathfrak{C}}(I[s+1]) := A_{\mathfrak{C}}(I[s])$.
		
\item		Otherwise, search for the least tuple $\bar a$ from $\omega$ such that the string $I[s+1]$ contains the following data: the tuple $\bar a$ forms a cycle of size $n$, where $n\in \{ 2,3\}$.
		\begin{itemize}
			\item If $n = 2$, then set $A_{\mathfrak{C}}(I[s+1]) :=1$.
			
			\item If $n = 3$, then $A_{\mathfrak{C}}(I[s+1]) :=2$.
			
			\item If there is no such $\bar a$, then define $A_{\mathfrak{C}}(I[s+1]) :=\ ?$.
\end{itemize}
		\end{itemize}
\end{itemize}

The described algorithm $A_{\mathfrak{C}}$ learns the family $\mathfrak{C}$: Suppose that an input $I$ encodes a structure $M$, which is isomorphic to either $G_1$ or $G_2$. Then there is a stage $s_0$ such that for any $s\geq s_0$, we have $A_{\mathfrak{C}}(I[s]) = A_{\mathfrak{C}}(I[s_0])$. Moreover, the conjecture $A_{\mathfrak{C}}(I[s_0])$ correctly identifies the isomorphism type of the graph $M$.

\smallskip

Our second learning problem, denoted as $\mathcal{P}_2$, is a generalization of the first one. Consider an \emph{infinite} family $\mathfrak{D}$, which consists of the following undirected graphs: for each $i\geq 1$, the graph $G_i$ contains infinitely many $(i+1)$-cycles and no other cycles.

The main features of $\mathcal{P}_2$ resemble those of $\mathcal{P}_1$: the learning domain of $\mathcal{P}_2$ is the family $\mathfrak{D}^*$ of \emph{all} presentations of the graphs in $\mathfrak{D}$;
 each informant $I$ provides both positive and negative information about any given graph in $\mathfrak{D}^*$; every conjecture is an element of the set $\omega \cup \{ ?\}$; a learner for $\mathcal{P}_2$ is an algorithm that learns, up to isomorphism, any graph in $\mathfrak{D}^*$; the
prior knowledge of $\mathcal{P}_2$ consists of the knowledge that the target graph is isomorphic to some graph from the family $\mathfrak{D}$.

  The intuition behind the desired learning algorithm $A_{\mathfrak{D}}$ is pretty straightforward: 

\begin{itemize}
\item Given a graph $H$, search for a cycle of some size $l+1$ inside it. When the first such cycle is found, start outputting the conjecture ``$H$ is a copy of $G_l$.''
\end{itemize}

The only technical problem of the algorithm $A_{\mathfrak{D}}$ is how to specify the hypothesis space of $\mathcal{P}_2$. Or, in other words:
\begin{center}
	\emph{How does one formally define the set of possible conjectures?}
\end{center}
We discuss two possible solutions of the problem, as they both seem to be pretty natural.


\smallskip

\emph{First Solution.} One can assume that, for any $m\in\omega$, the conjecture ``$m$'' means that ``$H\cong G_{m+1}$.'' 

This solution is similar to the so-called \emph{exact learning}, considered in the setting of computably enumerable (c.e.) languages (see, e.g., 	\cite{lange1993language,jain2011hypothesis}), where one assumes that the hypothesis space of the problem is precisely the class being learned
with the corresponding indexing. 
The exact learning algorithm $A^e_{\mathfrak{D}}$ is a straightforward modification of the algorithm $A_{\mathfrak{C}}$: 
\begin{itemize}
\item At a stage $s+1$, $A^e_{\mathfrak{D}}$ searches for the least tuple $\bar a$ such that the string $I[s+1]$ encodes the following data: the tuple $\bar a$ forms a cycle of some size $n\geq 2$. When such $\bar a$ is found, the algorithm starts outputting the conjecture ``$n-1$.''
\end{itemize}

One drawback of exact learning is that it can be computationally very hard to enumerate certain familiar families of computable structures, up to isomorphism: e.g., Goncharov and Knight~\cite{goncharov2002computable} proved that for the classes of  computable Boolean algebras, linear orders, and Abelian $p$-groups (we explore all such classes in Section \ref{sec:applications}) one cannot even hyperarithmetically enumerate their isomorphism types. This fact motivates the next solution.

\smallskip

\emph{Second Solution.} Fix a uniformly computable sequence $(\mathcal{M}_e)_{e\in\omega}$ of \emph{all} computable undirected graphs. W.l.o.g., one may assume that $\mathcal{M}_0\not\in\mathfrak{D}$ and $\mathcal{M}_{\langle i,0\rangle}\cong G_i$ for all $i\geq 1$. We assume that the conjecture ``$m$'' means that ``$H \cong \mathcal{M}_{m}$.''

This solution is similar to the so-called \emph{class-comprising learning} (see, e.g., \cite{lange1993language,jain2011hypothesis}), where one assumes that the hypothesis space of the problem should only contain the class being learned.

The class-comprising learning algorithm $A^{cc}_{\mathfrak{D}}$ works on an input $I$ as follows:
\begin{itemize}
	\item[(a)] First, as in the honest $A^{cc}_{\mathfrak{D}}$, we search for a cycle of some size $n\geq 2$. When the cycle is found, start outputting the conjecture ``$\langle n-1,0\rangle$.''
	
	\item[(b)] After that stage, assume that we find a finite piece of evidence (provided by $I$) showing that $G(I)\not\cong G_{n-1}$: e.g., we see that
	\begin{itemize}
		\item $G(I)$ contains a component of size at least $n+1$, or
		
		\item $G(I)$ contains a vertex of degree at least $3$, or
		
		\item $G(I)$ contains a cycle of size at most $n-1$.
	\end{itemize}
	Then we start outputting the conjecture ``$0$.''
\end{itemize}

\smallskip

The learning algorithms $A^{e}_{\mathfrak{D}}$ and $A^{cc}_{\mathfrak{D}}$ can be unified in a general framework as follows. One can consider an arbitrary superclass $\mathfrak{K}\supseteq \mathfrak{D}$. We assume that the class $\mathfrak{K}$ is \emph{uniformly enumerable}, i.e., there is a uniformly computable sequence of structures $(\mathcal{N}_e)_{e\in\omega}$ such that:
\begin{enumerate}
	\item Any structure from $\mathfrak{K}$ is isomorphic to some $\mathcal{N}_e$.
	
	\item For every $e$, $\mathcal{N}_e$ belongs to $\mathfrak{K}$.
\end{enumerate}
Then for a number $e\in\omega$, the conjecture ``$e$'' is interpreted as ``the input structure is isomorphic to $\mathcal{N}_e$.'' 



\subsection{Learning families of structures: Formal details} \label{subsect:formal}

We are now in a position of offering the formal definition of our learning paradigm: see Definition~\ref{definition:main} for the definition of the learning type $\Inf\Ex_{\cong}$.   

We begin with  the necessary formal preliminaries.

\smallskip

Let $L = \{ P^{n_0}_0, P^{n_1}_1, \dots, P^{n_k}_k\}$ be a relational signature. An $L$-\emph{informant} is a function
\[
	I \colon \omega \to (\omega^{n_0} \times \{ 0,1\}) \times (\omega^{n_1} \times \{ 0,1\}) \times \dots \times (\omega^{n_k} \times \{ 0,1\}).
\]
For a number $m$, the value $I(m)$ is treated as a $(k+1)$-tuple
\[
	I(m) = (I_0(m), I_1(m), \dots, I_k(m)),
\]
where $I_j(m) \in \omega^{n_j}\times \{ 0,1\}$. Let $\cont^{+}_j(I) := \{ \bar a\in \omega^{n_j} \,\colon (\bar a, 1) \in \range(I_j)\}$. That is, $\cont^{+}_j(I)$ is the set of all positive
examples of predicate $P_j$.

 The \emph{positive content} of the informant $I$ is the tuple
\[
	\cont^{+}(I) = (\cont^+_0(I), \cont^+_1(I),\dots, \cont^+_k(I)).
\]

Henceforth, for the sake of readability, we will often omit the arities of predicates. For an $L$-informant $I$ and an $L$-structure $\mathcal{S} = (\omega;P_0,P_1,\dots,P_k)$, we say that $I$ is an \emph{informant for} $\mathcal{S}$ if for every $i\leq k$, $\cont^+_i(I) = P_i$. By $\Inf(\mathcal{S})$ we denote the set of all informants for the structure $\mathcal{S}$. Observe that each informant, so defined, offers all positive, as well as all negative, data of the target structure.

If a signature $L$ contains functional symbols and/or constants, then one can use a standard convention from computable structure theory: by replacing functions with their graphs, we can treat any $L$-structure as a relational one. If a signature $L$ is clear from the context,  then we will talk about informants without specifying their prefix $L$-.

For a number $n$ and a function $f$ with $\dom(f)=\omega$, by $f[n]$ we denote the finite sequence $f(0),f(1),\dots,f(n-1)$.

A \emph{learner} is a function $M$ mapping initial segments of informants to conjectures (elements of $\omega \cup \set{?}$). The \emph{learning sequence} of a learner $M$ on an informant $I$ is the function $p\colon \omega\to \omega \cup \set{ ?}$ such that $p(n) = M(I[n])$ for every $n$.

Let $\sigma=(\sigma_1,\dots, \sigma_j,\dots, \sigma_k)$ be an initial part of an $L$-informant. By $\mathcal{A}_{\sigma}$ we denote the finite structure which is defined as follows:
The domain of $\mathcal{A}_{\sigma}$ is the greatest (under set-theoretic inclusion) set $D\subset\omega$ with the following properties: 
	\begin{itemize}
		\item[(a)] Every $x\in D$ is mentioned in $\sigma$, i.e., there are numbers $m<|\sigma|$, $j\leq k$, and a tuple $\bar a$ such that $x$ occurs in $\bar a$ and $\sigma_j(m)$ is equal to either $(\bar a, 0)$ or $(\bar a, 1)$.
		
		\item[(b)] If $j\leq k$ and $\bar b$ is a tuple from $D$ such that $|\bar b| = n_j$, then there is (the least) $m<|\sigma|$ with $\sigma_j(m) \in \{  (\bar b, 0), (\bar b, 1)\}$. 
	\end{itemize}
		
		The predicates on $\mathcal{A}_{\sigma}$ are recovered from the string $\sigma$ in a natural way: If $\sigma_j(m) = (\bar b, 1)$, then we set $\mathcal{A}_{\sigma} \models P_j(\bar b)$. Otherwise, we define $\mathcal{A}_{\sigma} \models \neg P_j(\bar b)$.
		
\smallskip		
		
	Informally speaking, the structure $\mathcal{A}_{\sigma}$ is constructed according to the following principle: We want to mine as much information from $\sigma$ as possible, but this information must induce a \emph{complete} diagram (of a finite structure). 
	
	Note that $\mathcal{A}_{\sigma}$ is allowed to be an empty $L$-structure. Nevertheless, if $I$ is an $L$-informant for a non-empty structure $\mathcal{B}$, then there is a stage $s_0$ such that for all $s\geq s_0$, we have $\mathcal{A}_{I[s]}\neq \emptyset$. Furthermore, it is clear that 
	\[
		\mathcal{A}_{I[s]} \subseteq \mathcal{A}_{I[s+1]} \text{ and } 
		\mathcal{B} = \bigcup_{s\in\omega}\mathcal{A}_{I[s]}.
	\]
	
	\begin{defn}
		Let $\mathfrak{K}$ be a class of $L$-structures. An \emph{effective enumeration} of the class $\mathfrak{K}$ is a function $\nu \colon \omega \to \mathfrak{K}$ with the following properties:
		\begin{enumerate}
			\item The sequence of $L$-structures $(\nu(e))_{e\in\omega}$ is uniformly computable.
			
			\item For any $\mathcal{A}\in\mathfrak{K}$, there is an index $e$ such that the structures $\mathcal{A}$ and $\nu(e)$ are isomorphic.
		\end{enumerate}
		In other words, the function $\nu$ effectively lists all isomorphism types from the class $\mathfrak{K}$ (possibly listing also other $L$-structures).
	\end{defn}
	
	Sometimes we abuse our notations: we assume that the notions ``enumeration'' and ``effective enumeration'' are synonymous. If $\nu$ and $\mu$ are two enumerations, then a new enumeration $\nu \oplus \mu$ is defined as follows.
	\[
		(\nu \oplus \mu)(2n) := \nu(n), \text{ and } (\nu \oplus \mu)(2n+1) := \mu(n).
	\]
	
	\begin{defn}
		Let $\nu$ be an effective enumeration of a class $\mathfrak{K}$, and let $\mathcal{A}$ be a structure from $\mathfrak{K}$. The \emph{index set of the structure $\mathcal{A}$ w.r.t.} $\nu$ is defined as follows:
		\[
			Ind(\mathcal{A};\nu) = \{ e\in\omega\,\colon \nu(e) \cong \mathcal{A}\}.
		\]
	\end{defn}
	
	We say that an effective enumeration $\nu$ is \emph{decidable} if the set 
	\[
		\{ (i,j)\,\colon \nu(i) \cong \nu(j)\}
	\]
	is computable. An effective enumeration $\nu$ is \emph{Friedberg}  if $\nu(i) \not\cong \nu(j)$ for all $i\neq j$ (Friedberg~\cite{friedberg1958three}  proved that there is an effective enumeration of all c.e.\ sets with no repetitions).
	
	\begin{remark}
		Note that any Friedberg enumeration is decidable. Moreover, if $\nu$ is a decidable enumeration of a class $\mathfrak{K}$, then for any $\mathcal{A} \in \mathfrak{K}$, its index set $Ind(\mathcal{A};\nu)$ is computable.
	\end{remark}
	
	Now we are ready to give the notion of informant learning:
	
	\begin{defn}\label{definition:main}
		Let $\mathfrak{K}$ be a class of $L$-structures, and let $\nu$ be an effective enumeration of $\mathfrak{K}$. Suppose that $\mathfrak{C}$ is a subclass of $\mathfrak{K}$. We say that $\mathfrak{C}$ is \emph{$\Inf\Ex_{\cong}[\nu]$-learnable} if there is a learner $M$ with the following property: If $I$ is an informant for a structure $\mathcal{A}\in\mathfrak{C}$, then there are $e$ and $s_0$ such that $\nu(e) \cong \mathcal{A}$ and $M(I[s]) = e$ for all $s\geq s_0$. In other words, in the limit, the learner $M$ learns all isomorphism types from $\mathfrak{C}$.
	\end{defn}
	
	Recall that the classes $\mathfrak{K}$ and $\mathfrak{C}$ are closed under isomorphisms. Hence, we emphasize that every structure $\mathcal{A}\in\mathfrak{C}$ has a computable copy, \emph{but} both the atomic diagram of $\mathcal{A}$ and an informant $I$ can have \emph{arbitrary} Turing degree. 
	
	We say that an $L$-structure $\mathcal{A}$ is \emph{$\Inf\Ex_{\cong}[\nu]$-learnable} if the class $\{ \mathcal{A}\}$ (or more formally, the class containing all isomorphic copies of $\mathcal{A}$) is $\Inf\Ex_{\cong}[\nu]$-le\-ar\-nable. Observe that every family $\mathfrak{C}$ consisting of a single isomorphism type $[\cA]_{\cong}$ is $\Inf\Ex_{\cong}[v]$-learnable: a learner just constantly outputs $\cA$.
	
	In this paper, we concentrate only on learning the isomorphism types of structures. Note that in~\cite{FKS-ta}, the learning notions were given for an arbitrary equivalence relation $\sim$ on a class $\mathfrak{K}$. 
	
		\begin{remark}
It might be natural to regard  the classical setting  of learning c.e.\ languages as a special case of our paradigm for learning computable structures. Yet, let us stress again that our framework is designed for modelling learning up to isomorphism (as opposed to the learning of a  given presentation of data, in Gold-style~\cite{Gold67}). So, since we assume that each structure considered has domain $\omega$, the only set that can be a target structure in our framework is $\omega$.
	\end{remark}
	


\subsection{Locking sequences}\label{subsect:locking}

The paper~\cite{FKS-ta} is focused on different versions of learning for various classes of equivalence structures. Here we briefly recap the results of~\cite{FKS-ta} on locking sequences, but now we formulate them for arbitrary classes of structures. The notion of a locking sequence was introduced by Blum and Blum~\cite{blum1975toward}.

\smallskip

We say that a finite sequence $\sigma$ \emph{describes a finite part} of an $L$-structure $\mathcal{A}$ if $\sigma$ is an initial segment of some $L$-informant for the structure $\mathcal{A}$. Note that since we are working with informant learning, $\sigma$ contains both positive and negative data about the structure $\mathcal{A}$.

\begin{defn}[{\cite[Definition 17]{FKS-ta}}]
	Suppose that $M$ is a learner and $\mathcal{A}$ is an $L$-stru\-c\-ture. A sequence $\sigma$ describing a finite part of $\mathcal{A}$ is a \emph{weak informant locking sequence of $M$ on $\mathcal{A}$} if for every $\tau \supseteq \sigma$ describing a finite part of $\mathcal{A}$, we have $M(\tau) = M (\sigma)$.
\end{defn}

\begin{thm}[{\cite[Theorem~18]{FKS-ta}}]\label{theo:weak-lock-001}
	Let $\nu$ be an effective enumeration of a class $\mathfrak{K}$, and let $\mathcal{A}$ be a structure from $\mathfrak{K}$. Suppose that  a learner $M$ $\Inf\Ex_{\cong}[\nu]$-learns the structure $\mathcal{A}$. Let $\sigma_0$ be a sequence which describes a finite part of $\mathcal{A}$. Then there is a finite sequence $\sigma\supseteq \sigma_0$ such that $\sigma$ is a weak informant locking sequence of $M$ on $\mathcal{A}$. Furthermore, $\nu(M(\sigma)) \cong \mathcal{A}$.
\end{thm}
\begin{proof}[Proof Sketch]
	Towards a contradiction, suppose that there is $\sigma_0$ with no weak locking sequence $\sigma\supseteq \sigma_0$. Then for any $\sigma \supseteq \sigma_0$ describing a finite part of $\mathcal{A}$, there is a string $ext(\sigma) \supset\sigma$ such that $ext(\sigma)$ also describes a finite part of $\mathcal{A}$, and $M({ext(\sigma)}) \neq M ({\sigma})$. 
	
	Fix an informant $I$ for $\mathcal{A}$. Then one can  produce a new informant $I'$ for $\mathcal{A}$ such that the learner $M$ does not correctly converge on $I'$: Just ``alternate'' between the data given by $I$ and ``bad'' extensions $ext(\sigma)$, in an appropriate way.
\end{proof}

\begin{defn}[{\cite[Definition~19]{FKS-ta}}]\label{def:inf-lock}
	Let $M$ be a learner and $\mathcal{A}$ be an $L$-structure. We say that $M$ is \emph{informant locking} on $\mathcal{A}$ if for every informant $I$ for $\mathcal{A}$, there is an $n$ such that $I[n]$ is a weak informant locking sequence for $M$ on $\mathcal{A}$. 
		Assume that a class $\mathfrak{A}$ is $\Inf\Ex_{\cong}[\nu]$-learnable. A learner $M$ which $\Inf\Ex_{\cong}[\nu]$-le\-arns $\mathfrak{A}$ is \emph{informant locking} if it is informant locking for every $\mathcal{A}\in\mathfrak{A}$.
\end{defn}

\begin{thm}[see Theorem~20 in {\cite{FKS-ta}}]\label{theo:weak-lock-002}
	If a class $\mathfrak{A}$ is $\Inf\Ex_{\cong}[\nu]$-le\-ar\-na\-ble, then there is an informant locking learner $M$ which $\Inf\Ex_{\cong}[\nu]$-learns $\mathfrak{A}$.
\end{thm}


\subsection{Infinitary formulas}\label{subsect:infinitary}

Suppose that $X \subseteq \omega$ is an oracle, and $\alpha$ is an $X$-computable non-zero ordinal. Following  Chapter~7 of~\cite{AK00}, we describe the class of $X$-computable infinitary $\Sigma_{\alpha}$ formulas (or $\Sigma^c_{\alpha}(X)$ formulas, for short) in a signature $L$.
\begin{itemize}
	\item[(a)] $\Sigma^c_0(X)$  and $\Pi^c_0(X)$ formulas are quantifier-free first-order $L$-formulas.
	
	\item[(b)] A $\Sigma_{\alpha}^c(X)$ formula $\psi(x_0,\dots,x_m)$ is an $X$-computably enumerable ($X$-c.e.) disjunction
	\[
		\underset{i\in I}{\bigvee\skipmm{4}\bigvee} \exists \bar y_i \xi_i(\bar x, \bar y_i),
	\]
	where each $\xi_i$ is a $\Pi^c_{\beta_i}(X)$ formula, for some $\beta_i< \alpha$.
	
	\item[(c)] A $\Pi_{\alpha}^c(X)$ formula $\psi(\bar x)$ is an $X$-c.e. conjunction
	\[
		\underset{i\in I}{\bigwedge\skipmm{4}\bigwedge} \forall \bar y_i \xi_i(\bar x, \bar y_i),
	\]
	where each $\xi_i$ is a $\Sigma^c_{\beta_i}(X)$ formula, for some $\beta_i < \alpha$.
\end{itemize}

In the paper, we mainly work with $\Sigma^{c}_{\alpha}(X)$ formulas for finite ordinals $\alpha$ (even more, for $\alpha \leq 2$). Henceforth, in this section we assume that $\alpha = n$ is a natural number.

\emph{Infinitary $\Sigma_n$ formulas} (or \emph{$\Sigma^{\infi}_{n}$ formulas}, for short) are defined in the same way as above, modulo the following modification: infinite disjunctions and conjunctions are not required to be $X$-c.e. It is clear that a formula $\psi$ is logically equivalent to a $\Sigma^{\infi}_{n}$ formula iff $\psi$ is equivalent to a $\Sigma^c_{n}(X)$ formula for some oracle $X$. A similar fact holds for $\Pi^{\infi}_{n}$ formulas. For more details on infinitary formulas, we refer the reader to~\cite{AK00}.

 As usual, the \emph{$\Sigma^{\infi}_n$-theory}
 of an $L$-structure $\mathcal{S}$ is the set
\[
	\Sigma^{\infi}_n \text{-} Th(\mathcal{S}) = \{ \psi\,\colon \psi \text{ is a } \Sigma^{\infi}_n \text{ sentence true in } \mathcal{S} \}.
\]


\section{Learning from informant, and infinitary $\Sigma_2$-theories}\label{sec:main_characterization}

In this section, we offer a model-theoretic characterization of what families of structures are $\Inf\Ex_{\cong}[\nu]$-learnable: Informally speaking, we show that a family of structures $\mathfrak{K}$ is $\Inf\Ex_{\cong}[\nu]$-learnable if and only if the (isomorphism types of) structures from $\mathfrak{K}$ can be distinguished in terms of their $\Sigma^{\infi}_2$-the\-o\-ries. 

Suppose that $\mathfrak{K}_0$ is a class of $L$-structures, and $\nu$ is an effective enumeration of the class $\mathfrak{K}_0$.

\begin{thm}\label{thm:Sigma_2-theories}
	Let $\mathfrak{K} = \{ \mathcal{B}_i\,\colon i\in\omega\}$ be a family of structures  
	such that $\mathfrak{K}\subseteq \mathfrak{K}_0$, and the structures $\mathcal{B}_i$ are infinite and pairwise non-isomorphic.	
	Then the following conditions are equivalent:
\begin{enumerate}
	\item The class $\mathfrak{K}$ is $\Inf\Ex_{\cong}[\nu]$-learnable.

	\item There is a sequence of $\Sigma^{\infi}_2$ sentences $\{ \psi_i\,\colon i\in\omega\}$ such that for all $i$ and $j$, we have $\mathcal{B}_j\models \psi_i$ if and only if $i=j$.
\end{enumerate}	
\end{thm}

 Theorem~\ref{thm:Sigma_2-theories} talks about classes $\mathfrak{K}$ which contain infinitely many isomorphism types. Nevertheless, one can easily formulate (and prove) an analogous result for classes with only finitely many isomorphism types: Just work with a family $\mathfrak{K} = \{ \mathcal{B}_0, \mathcal{B}_1, \dots, \mathcal{B}_n\}$ and the corresponding finite sequence of $\Sigma^{\infi}_2$ sentences $\{  \psi_0, \psi_1, \dots, \psi_n\}$.

\begin{remark}
The statement of Theorem~\ref{thm:Sigma_2-theories} is similar to a result due to Martin and Osherson~\cite[p.~79, Corollary~(52)]{Mar-Osh:b:98}. Yet, our proof is novel and based on a technique introduced by Knight, Miller, and Vanden Boom~\cite{KMV07} in the context of Turing computable embeddings. A main upshot of our approach is that it provides an upper bound for the Turing complexity of the learners (Corollary~\ref{coroll:complexity}), which will be crucial, in Section \ref{sec:applications}, for analyzing the learnability of familiar classes of structures.
\end{remark}

\smallskip

The proof of Theorem~\ref{thm:Sigma_2-theories} is organized as follows. Section~\ref{subsect:TC-emb} discusses the necessary preliminaries on Turing computable embeddings, which constitute one of the main ingredients of the proof. In Section~\ref{subsect:connect}, we give a result (Proposition~\ref{prop:exists-tc-emb}) which provides a connection between $\Inf\Ex_{\cong}$-le\-ar\-na\-bi\-li\-ty and Turing computable embeddings. Section~\ref{subsect:proof-finish} finishes the proof. Section~\ref{subsect:further-disc} discusses some further questions related to the proof.

\subsection{Turing computable embeddings} \label{subsect:TC-emb}

When we are working with Turing computable embeddings, we consider structures $\mathcal{S}$ such that the domain of $\mathcal{S}$ is an \emph{arbitrary} subset of $\omega$. In contrast, recall that our learning paradigm applies only to structures with domain equal to $\omega$. As before, any considered class of structures is closed under isomorphisms, modulo the domain restrictions.

Let $\mathfrak{K}_0$ be a class of $L_0$-structures, and $\mathfrak{K}_1$ be a class of $L_1$-structures.

\begin{defn}[{\cite{CCKM04,KMV07}}] \label{def:tc-embedding}
  A Turing operator $\Phi=\Phi_e$ is a \emph{Turing computable embedding} of $\mathfrak{K}_0$ into $\mathfrak{K}_1$, denoted by $\Phi\colon \mathfrak{K}_0 \leq_{tc} \mathfrak{K}_1$, if $\Phi$ satisfies the following:
  \begin{enumerate}
  	\item For any $\mathcal{A}\in \mathfrak{K}_0$, the function $\Phi^{D(\mathcal{A})}_e$ is the characteristic function of the atomic diagram of a structure from $\mathfrak{K}_1$. This structure is denoted by $\Phi(\mathcal{A})$.

  	\item For any $\mathcal{A},\mathcal{B}\in \mathfrak{K}_0$, we have $\mathcal{A}\cong\mathcal{B}$ if and only if $\Phi(\mathcal{A}) \cong \Phi(\mathcal{B})$.
  \end{enumerate}	
\end{defn} 

The term ``Turing computable embedding'' is often abbreviated as \emph{$tc$-em\-bed\-ding}. One of the important results in the theory of $tc$-embeddings is the following. Recall that $\omega^{CK}_1$ denotes the smallest ordinal which is noncomputable.

\begin{thm}[Pullback Theorem; Knight, Miller, and Vanden Boom~\cite{KMV07}] \label{thm:Pullback}
  Suppose that $\mathfrak{K}_0 \leq_{tc} \mathfrak{K}_1$ via a Turing operator $\Phi$. Then for any computable infinitary sentence $\psi$ in the signature of $\mathfrak{K}_1$, one can effectively find a computable infinitary sentence $\psi^{\star}$ in the signature of $\mathfrak{K}_0$ such that for all $\mathcal{A}\in \mathfrak{K}_0$, we have $\mathcal{A} \models \psi^{\star}$ if and only if $\Phi(\mathcal{A}) \models \psi$. Moreover, for a non-zero $\alpha <\omega^{CK}_1$, if $\psi$ is a $\Sigma^c_{\alpha}$ formula ($\Pi^c_{\alpha}$ formula), then so is $\psi^{\star}$.
\end{thm}

An analysis of the proof of Theorem~\ref{thm:Pullback} shows that this result admits a full relativization as follows. 

Fix an oracle $X \subseteq \omega$. In a natural way, a \emph{Turing $X$-relativized operator} $\varphi_{e,X}$ can be defined as follows: for a set $Z\subseteq \omega$ and a natural number $k$, let
\[
	\Phi^{Z}_{e,X}(k) := \Phi^{Z\oplus X}_e(k),
\]
where $Z\oplus X$ denotes the usual join of $Z$ and $X$, i.e., $Z\oplus X= \{ 2x : x \in Z\}\cup \{2x+1: x \in X\}$.

We often denote a Turing $X$-relativized operator as $\Phi_{[X]}$.
Informally speaking, one can identify a Turing $X$-relativized operator with a Turing machine which has three tapes: the input tape (on which the machine is allowed to work), the output tape, and the oracle tape, where the oracle tape \emph{always} contains the characteristic function of $X$.

In a straightforward way, one can use the notion of a Turing $X$-re\-la\-ti\-vi\-zed operator to introduce \emph{Turing $X$-computable embeddings}. If there is a Turing $X$-computable embedding from $\mathfrak{K}_0$ into $\mathfrak{K}_1$, then we write $\mathfrak{K}_0 \leq_{tc}^X \mathfrak{K}_1$.

One can obtain the following consequence of Theorem~\ref{thm:Pullback}. 

\begin{corollary}[Relativized Pullback Theorem] \label{corol:Pullback}
	Suppose that $X\subseteq \omega$, and $\mathfrak{K}_0 \leq_{tc}^X \mathfrak{K}_1$ via an operator $\Phi_{[X]}$. Then for any $X$-computable infinitary sentence $\psi$ in the signature of $\mathfrak{K}_1$, one can find, effectively with respect to $X$, an $X$-computable infinitary sentence $\psi^{\star}$ in the signature of $\mathfrak{K}_0$ such that for all $\mathcal{A}\in \mathfrak{K}_0$, we have $\mathcal{A} \models \psi^{\star}$ if and only if $\Phi_{[X]}(\mathcal{A}) \models \psi$. Furthermore, for a non-zero $\alpha <\omega^X_1$, if $\psi$ is a $\Sigma^{c}_{\alpha}(X)$ formula ($\Pi^{c}_{\alpha}(X)$ formula), then so is $\psi^{\star}$.
\end{corollary}



\subsection{Connecting $\Inf\Ex_{\cong}$-learnability and $tc$-embeddings} \label{subsect:connect}

Let $L$ be a finite signature, and $\mathfrak{K}_0$ be a class of $L$-structures. Let $\nu$ be an effective enumeration of the class $\mathfrak{K}_0$.

Suppose that $\mathfrak{K} = \{ \mathcal{B}_i \,\colon i\in\omega\}$ is a family of $L$-structures with the following properties:
\begin{itemize}
	\item[(a)] $\mathfrak{K}$ is a subclass of $\mathfrak{K}_0$. All $\mathcal{B}_i$ are infinite and pairwise non-isomorphic.
	
	\item[(b)] There is a learner $M$ which $\Inf\Ex_{\cong}[\nu]$-learns the class $\mathfrak{K}$.
\end{itemize}

We choose the oracle $X$ as follows:
\begin{equation} \label{equ:oracle_choosing}
	X := M \oplus \{ \langle i,k\rangle\,\colon i\in\omega,\ k\in Ind(\mathcal{B}_i;\nu) \} \oplus \{ j \,\colon \exists i (  j\in Ind(\mathcal{B}_i;\nu))\}.
\end{equation}

Consider a signature 
\[
	L_{st} := \{ \leq\} \cup \{ P_i\,\colon i\in\omega\},
\] 
where every $P_i$ is a unary relation. For $i\in\omega$, we define an $L_{st}$-structure $\mathcal{S}_i$ as follows: All $P_j$ are disjoint. For $j\neq k$, if $x\in P_j$ and $y\in P_k$, then $x$ and $y$ are incomparable under $\leq$. Every $P_j$, $j\neq i$, contains a $\leq$-structure isomorphic to the order type $\eta$ of the rationals. The relation $P_i$ contains a copy of $1+\eta$.

Let $\mathfrak{K}_{st}$ denote the class $\{ \mathcal{S}_i\,\colon i\in\omega\}$.

\begin{prop}\label{prop:exists-tc-emb}
	There is a Turing $X$-computable embedding $\Phi_{[X]}$ from $\mathfrak{K}$ into $\mathfrak{K}_{st}$ such that for any $i\in\omega$, we have $\Phi_{[X]}(\mathcal{B}_i) \cong \mathcal{S}_i$.
\end{prop}
\begin{proof}
	Let $\mathcal{C}$ be a structure such that $\mathcal{C}$ is isomorphic to some $\mathcal{B}_i$, and $dom(\mathcal{C}) \subseteq \omega$. 
	
	It is not hard to show that there is a Turing operator $\Psi$ with the following property: If $\mathcal{E}$ is a countably infinite $L$-structure with $dom(\mathcal{E})\subseteq\omega$, then $\Psi^{D(\mathcal{E})}$ is the atomic diagram of a structure $\mathcal{E}_1$ such that $dom(\mathcal{E}_1)=\omega$ and $\mathcal{E}_1$ is $D(\mathcal{E})$-computably isomorphic to $\mathcal{E}$.
	
	The existence of the operator $\Psi$ implies that w.l.o.g., we may assume that the domain of our $\mathcal{C}$ is equal to $\omega$. For simplicity, we assume that $L = \{ Q_0, Q_1, \dots, Q_l\}$, where each $Q_i$ has arity $i+1$. For $i\leq l$, fix a computable bijection $\gamma_i\colon \omega\to \omega^{i+1}$.
	
	We describe the construction of the $L_{st}$-structure $\Phi_{[X]}(\mathcal{C})$. First, define an $L$-informant $I^{\mathcal{C}}$ as follows. For $i\leq l$ and $m\in\omega$, set:
	\[
		I^{\mathcal{C}}_i(m) = 
		\begin{cases}
			(\gamma_i(m), 1), & \text{if } \mathcal{C}\models Q_i(\gamma_i(m)),\\
			(\gamma_i(m), 0), & \text{if } \mathcal{C}\models \neg Q_i(\gamma_i(m)).
		\end{cases}
	\]
	
	Fix a computable copy $\mathcal{M}$ of the ordering $\eta$, and choose a computable descending sequence $q_0>_{\mathcal{M}} q_1 >_{\mathcal{M}} q_2 >_{\mathcal{M}} \dots$.
	
	The construction of the structure $\mathcal{E} = \Phi_{[X]}(\mathcal{C})$ proceeds in stages.
	
	\emph{Stage 0.} Put inside every $P_j^{\mathcal{E}}$, $j\in\omega$, a computable copy of the interval $(q_0;\infty)_{\mathcal{M}}$.
	
	\emph{Stage $s+1$.} Recall that the learner $M$ $\Inf\Ex_{\cong}[\nu]$-learns the class $\mathfrak{K}$. Compute the value $t := M(I^{\mathcal{C}}[s+1])$. Using the oracle $X$, one can find whether the number $t$ is a $\nu$-index for some $\mathcal{B}_j$, $j\in\omega$.
	
	If $t$ is not a $\nu$-index for any $\mathcal{B}_j$, then extend every $P^{\mathcal{E}}_k$, $k\in\omega$, to a copy of $(q_{s+1};\infty)_{\mathcal{M}}$.
	
	Otherwise, assume that $t$ is an index for $\mathcal{B}_j$. If $P^{\mathcal{E}}_j[s]$ has the least element, then do not change $P^{\mathcal{E}}_j[s]$. If $P^{\mathcal{E}}_j[s]$ has no least element, then define $P^{\mathcal{E}}_j[s+1]$ as a copy of the interval $[q_{s+1};\infty)_{\mathcal{M}}$. Note that this interval is isomorphic to $1+\eta$. In any case, extend every $P^{\mathcal{E}}_k[s]$, $k\neq j$, to a copy of the open interval $(q_{s+1};\infty)_{\mathcal{M}}$.
	
	This concludes the description of the construction. It is not hard to show that the construction gives a Turing $X$-computable operator $\Phi_{[X]}$. Moreover, if the input structure $\mathcal{C}$ is isomorphic to $\mathcal{B}_i$, then there is a stage $s_0$ such that for any $s\geq s_0$, we have $M(I^{\mathcal{C}}[s])=M(I^{\mathcal{C}}[s_0])$ is a $\nu$-index of the structure $\mathcal{B}_i$. Hence, $P_i^{\Phi_{[X]}(\mathcal{C})}$ contains a copy of $1+\eta$, and for every $j\neq i$, $P_j^{\Phi_{[X]}(\mathcal{C})}$ copies $\eta$. Thus, $\Phi_{[X]}(\mathcal{C})$ is isomorphic to $\mathcal{S}_i$. 
	
	Proposition~\ref{prop:exists-tc-emb} is proved.
\end{proof}


\subsection{Proof of Theorem~\ref{thm:Sigma_2-theories}} 
\label{subsect:proof-finish}

\begin{proof}
	$(1) \Rightarrow (2)$: 
	Choose an oracle $X$ according to Equation~(\ref{equ:oracle_choosing}). By Proposition~\ref{prop:exists-tc-emb}, there is a Turing $X$-computable embedding 
	\[
		\Phi_{[X]} \colon\mathfrak{K} \leq^X_{tc} \mathfrak{K}_{st}
	\] 
	such that $\Phi_{[X]}(\mathcal{B}_i)$ is a copy of $\mathcal{S}_i$. 
	
	Consider an $\exists\forall$-sentence in the signature $L_{st}$
	\[
		\xi_i := \exists x \forall y [ P_i(y) \rightarrow (x\leq y)].
	\]
	Note that $\mathcal{S}_j\models \xi_i$ if and only if $i=j$. By Corollary~\ref{corol:Pullback}, we obtain a sequence of $X$-computable infinitary $\Sigma_2$ sentences $(\xi^{\star}_i)_{i\in\omega}$. Clearly, this sequence has the desired properties.
	
	\smallskip
	
	$(2) \Rightarrow (1)$: W.l.o.g., for all $i$, assume that 
\[
\psi_i:=\exists x_1,\ldots,x_{n_i} \underset{j\in J_i}{\bigwedge\skipmm{4.5}\bigwedge} \forall  y_1,\ldots,y_{m_{i,j}} 
\phi_{i,j}(x_1,\dots,x_{n_i},y_1,\dots,y_{m_{i,j}}),
\]
where every $\phi_{i,j}$ is a quantifier-free formula.

Let $\C$ be a finite structure, and $i\in\omega$. We say that 
the formula $\psi_i$ is $\C$-\emph{compatible} via a tuple $\bar a\in\omega^{n_i}$ if within $\dom(\C)$ there is no pair $(j,\bar b)$, with $j \in J_i$ and $\bar b\in\omega^{m_{i,j}}$, such that $\C\models \neg\phi_{i,j}(\bar a,\bar  b)$. 

We fix a sequence $(e_i)_{i\in\omega}$ such that for every $i$, the structure $\nu(e_i)$ is a copy of $\mathcal{B}_i$.

A learner $M$ for the class $\mathfrak{K}$ can be arranged as follows:
Suppose that $M$ reads a string $\sigma$, which is an initial part of some $L$-informant. Then we search for the least pair $\langle i,\bar a\rangle$ such that the formula $\psi_i$ is $\mathcal{A}_{\sigma}$-compatible via the tuple $\bar a$. If the pair $\langle i,\bar a \rangle$ is found, then set $M(\sigma):= e_i$. Otherwise, define $M(\sigma):=0$.


\smallskip

\textbf{Verification.} 
Fix $j\in\omega$. Let $I$ be an informant for the structure $\mathcal{B}_j$. Recall that $\mathcal{B}_j = \bigcup_{s\in\omega} \mathcal{A}_{I[s]}$ and $\mathcal{A}_{I[s]}\subseteq \mathcal{A}_{I[s+1]}$.

We note the following simple fact: Suppose that a formula $\psi_i$ is not $\mathcal{A}_{I[t_0]}$-compatible via a tuple $\bar d$. Then for any $t\geq t_0$, $\psi_i$ also cannot be $\mathcal{A}_{I[t]}$-compatible via $\bar d$.

Recall that $\mathcal{B}_j\models \psi_i$ if and only if $i=j$. Hence, there exists the least tuple $\bar a\in \omega^{n_j}$ with the following property: there is a stage $s_0$ such that for every $s\geq s_0$, the formula $\psi_j$ is $\mathcal{A}_{I[s]}$-compatible via the tuple $\bar a$. Furthermore, it is not difficult to see that
\[
	\mathcal{B}_j\models \neg\psi_i \ \Leftrightarrow\ 
	(\forall \bar c \in\omega^{n_i})(\exists s_1)(\psi_i \text{ is not } \mathcal{A}_{I[s_1]} \text{-compatible via } \bar c).
\]
Hence, for every number $\langle k, \bar c\rangle < \langle j,\bar a\rangle$, there is a stage $t_1$ such that for any $t\geq t_1$, the formula $\psi_k$ is not $\mathcal{A}_{I[t]}$-compatible via $\bar c$. This means that there is $t^{\star}$, such that the current conjecture $M(I[t^{\star}])$ is correct (i.e., $\nu(M(I[t^{\star}]))$ is a copy of $\mathcal{B}_j$), and our learner $M$ does not change its mind after the stage $t^{\star}$.


Therefore, the class $\mathfrak{K}$ is $\Inf\Ex_{\cong}[\nu]$-learnable by the learner $M$. This concludes the proof of Theorem~\ref{thm:Sigma_2-theories}.
\end{proof}

\subsection{Further discussion} \label{subsect:further-disc}
	We note that it would be interesting to attack the following question: If a class $\mathfrak{K} = \{ \mathcal{B}_i\,\colon i\in\omega\}$ is $\Inf\Ex_{\cong}[\nu]$-learnable, could one construct \emph{explicitly} some sequence $\{\psi_i\,\colon i\in\omega\}$ of $\Sigma^{\mathrm{inf}}_2$-sentences distinguishing the structures $\mathcal{B}_i$?  
	
	To our best knowledge, it seems that our proof of Theorem~\ref{thm:Sigma_2-theories} does not provide such a construction. Furthermore, even for the case when a class $\mathfrak{K}$ is learnable by a computable learner, it is quite hard to give a nice description of the properties of $\mathcal{B}_i$ expressed by our formulas $\psi_i$.
	
	Indeed, suppose that $\{ \mathcal{B}_i\,\colon i\in\omega\}$ is learnable by a computable learner. Then in general, the oracle $X$ from Eq.~(\ref{equ:oracle_choosing}) can be noncomputable. A simple example of a noncomputable $X$ is provided by the family containing two isomorphism types of linear orders: $\mathcal{B}_0 = \omega$ and $\mathcal{B}_1 = \omega^{\ast}$. Suppose that we consider the standard effective enumeration $\nu$, which enumerates all computable structures in the signature $\{ \leq\}$. Then it is not hard to show that both index sets $Ind(\mathcal{B}_0;\nu)$ and $Ind(\mathcal{B}_1;\nu)$ are $\Pi^0_3$-complete. Therefore, the corresponding oracle $X$ is not even $\mathbf{0}^{(2)}$-c.e., let alone computable.
	
	For the sake of simplicity, assume that for a particular class $\mathfrak{K}$, the obtained oracle $X$ is computable. Even in this case, there are further complications. We illustrate these problems by an informal ``toy'' example. The example is, in a sense, a simplified version of the proof $(1) \Rightarrow (2)$ of Theorem~\ref{thm:Sigma_2-theories}.
	
	Consider a class $\mathfrak{K}$ consisting of two computable undirected graphs: 
	\begin{enumerate}
		\item The graph $G_1$ contains infinitely many isolated nodes and one cycle of size $2k+3$ for each $k\in\omega$.
		
		\item The graph $G_2$ has infinitely many isolated nodes and one cycle of size $2k+4$ for each $k$.
	\end{enumerate}
	We define a computable learner $M$, which acts according to the following rules:
	\begin{itemize}
		\item Given an input graph $H$, $M$ searches for the least natural numbers $k_0<l_0$ such that $H$ contains the edge $(k_0,l_0)$.
		
		\item When $k_0$ and $l_0$ are found, $M$ searches for a cycle of size $n\in\{ 2k_0+3, 2k_0+4\}$ inside $H$. If $n=2k_0+3$, then $M$ says that $H$ is isomorphic to $G_1$. If $n=2k_0+4$, then $M$ says $H \cong G_2$.
	\end{itemize}
	
	Consider finite undirected graphs $F_1$ and $F_2$ such that $\mathrm{dom}(F_1) = \mathrm{dom}(F_2) = \{ 0,1\}$, the two nodes of $F_1$ are isolated, and $F_2$ contains an edge between $0$ and $1$. By employing the learner $M$, one can proceed similarly to Proposition~\ref{prop:exists-tc-emb} and construct a Turing computable embedding
	\[
		\Phi\colon \{ G_1, G_2\} \leq_{tc} \{ F_1, F_2\}
	\]
	such that $\Phi$ satisfies a stronger condition: for each $i\in\{ 1,2\}$, if $H$ is an isomorphic copy of $G_i$, then $\Phi(H)$ \emph{equals} $F_i$.
	
	Consider two existential sentences in the signature of graphs:
	\[
		\xi_1 = \exists x \exists y[ x\neq y \,\&\, \neg \textnormal{Edge}(x,y)] \text{ and } \xi_2 = \exists x \exists y[ x\neq y \,\&\, \textnormal{Edge}(x,y)].
	\]
	One can apply the proof of the Pullback Theorem for $\Sigma^c_1$-sentences (see Special Case on p.~905 of~\cite{KMV07}). The $tc$-embedding $\Phi$ induces $\Sigma^c_1$-sentences $\xi^{\star}_1$ and $\xi^{\star}_2$ such that
	\begin{equation}\label{equ:graphs-example}
		G_1 \models \xi^{\star}_1 \& \neg \xi^{\star}_2 \text{ and } G_2 \models \neg \xi^{\star}_1 \& \xi^{\star}_2.
	\end{equation}
	
	An analysis of the proof of~\cite{KMV07} shows that for \emph{this particular} $tc$-embedding $\Phi$, we have:
	\begin{enumerate}
		\item $\xi^{\star}_1$ is an infinite disjunction, which includes formulas
		\[
			\theta_{2k+3} = \exists x_1 \exists x_2 \dots \exists x_{2k+3}[x_i \text{-s form a cycle of size } 2k+3]
		\]
		(and possibly some other $\exists$-formulas).
		
		\item Similarly, $\xi^{\star}_2$ includes a disjunction of formulas
		\[
			\theta_{2k+4} = \exists x_1 \exists x_2 \dots \exists x_{2k+4}[x_i \text{-s form a cycle of size } 2k+4].
		\]
	\end{enumerate}
	
	On the other hand, it is clear that one can replace these formulas $\xi^{\star}_1$ and $\xi^{\star}_2$ with $\xi^{\#}_1 = \theta_3$ and  $\xi^{\#}_2 = \theta_4$, while preserving the property~(\ref{equ:graphs-example}).
	
	The described example shows that in general, the concrete formulas $\xi^{\star}_i$, built in Theorem~\ref{thm:Sigma_2-theories}, depend on the choice of $tc$-embedding $\Phi$. Thus, it is hard to say how the formulas are related to familiar algebraic properties of the original structures $\mathcal{B}_i$.
	
	In conclusion, we note that Theorem~\ref{thm:Sigma_2-theories} \emph{does not use} the full strength of the Pullback Theorem: it is sufficient to employ Pullback only for finitary formulas of the form $\xi = \exists \bar x \forall \bar y \theta(\bar x, \bar y)$, where $\theta$ is quantifier-free.
	Nevertheless, it seems that the proof of such restricted version of the Pullback Theorem still requires developing essentially the same forcing machinery as for the general form.


\section{Applications of the main result}\label{sec:applications}

The first application gives an upper bound for the Turing complexity of learners. A straightforward analysis of the proof of Theorem~\ref{thm:Sigma_2-theories} provides us with the following:

\begin{corollary}\label{coroll:complexity}
	Let $X\subseteq \omega$ be an oracle. Let $\mathfrak{K}_0$ be a class of countably infinite $L$-structures, and $\nu$ be an effective enumeration of $\mathfrak{K}_0$. Assume that either $I = \omega$, or $I$ is a finite initial segment of $\omega$. Consider a subclass $\mathfrak{K} = \{ \mathcal{B}_i \,\colon i\in I\}$ inside $\mathfrak{K}_0$. Assume that
	\begin{itemize} 
		\item[(i)] There is uniformly $X$-com\-pu\-table sequence of $\Sigma^c_2(X)$ sentences $(\psi_i)_{i\in I}$ such that:
		\[
			\mathcal{B}_j\models \psi_i \ \Leftrightarrow\ i=j.
		\]
		
		\item[(ii)] There is an $X$-computable sequence $(e_i)_{i\in I}$ such that $\nu(e_i)\cong \mathcal{B}_i$ for all $i$. Note that if  the set $I$ is finite, then one can always choose this sequence in a computable way.
	\end{itemize}
	Then the class $\mathfrak{K}$ is $\Inf\Ex_{\cong}[\nu]$-learnable via an $X$-computable learner.
\end{corollary}

The rest of the section discusses applications of Theorem~\ref{thm:Sigma_2-theories} and Corollary~\ref{coroll:complexity} to some familiar classes of algebraic structures.

\subsection{Simple examples of learnable classes}

Here we give two examples of learnable infinite families. 

The first one deals with distributive lattices. We treat lattices as structures in the signature $L_{\mathrm{lat}} := \{ \vee, \wedge\}$.

Selivanov~\cite{Selivanov88} constructed a uniformly computable family $\{ \mathcal{D}_i\,\colon i\in\omega\}$ of finite distributive lattices with the following property: If $i\neq j$, then there is no isomorphic embedding from $\mathcal{D}_i$ into $\mathcal{D}_j$ (see Figure~1).

For $i\in\omega$, we define a countably infinite poset $\mathcal{B}_i$. Informally speaking, $\mathcal{B}_i$ is a direct sum of the lattice $\mathcal{D}_i$ and the linear order $\omega$. More formally, we set:
\begin{itemize}
	\item $dom(\mathcal{B}_i) = \{ \langle x,0\rangle \,\colon x\in \mathcal{D}_i\} \cup \{ \langle y,1\rangle\,\colon y\in\omega\}$.
	
	\item We always assume that $\langle x,0\rangle \leq \langle y,1\rangle$. The ordering of the elements $\langle x,0\rangle$ is induced by $\mathcal{D}_i$. We have $\langle y,1\rangle \leq \langle z,1\rangle$ if and only if $y \leq_{\omega} z$.
\end{itemize}
It is not hard to show that $\mathcal{B}_i$ is a distributive lattice, thus, we will treat $\mathcal{B}_i$ as an $L_{\mathrm{lat}}$-structure.

Let $\mathfrak{K}_{lat}$ denote the class $\{ \mathcal{B}_i\,\colon i\in\omega\}$. It is clear that one can build a Friedberg effective enumeration $\nu_{lat}$ as follows: just define $\nu_{lat}(i)$ as a natural computable copy of $\mathcal{B}_i$.

\begin{figure} \label{figure:lattices}
	\includegraphics{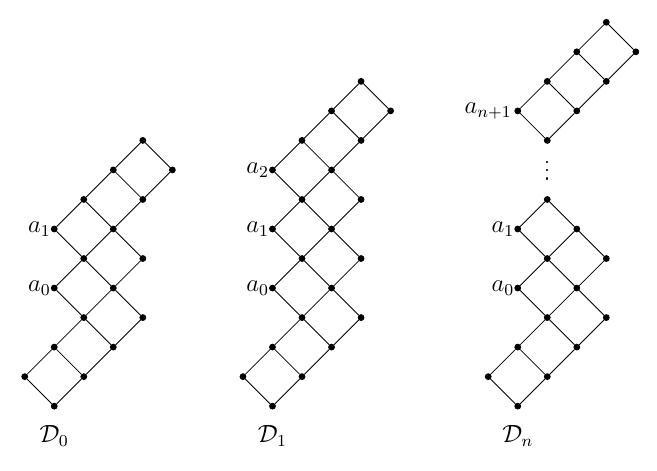}
	\caption{Finite lattices $\mathcal{D}_i$, $i\in\omega$.}
\end{figure}

\begin{prop} \label{prop:lattices}
	The class $\mathfrak{K}_{lat}$ is $\Inf\Ex_{\cong}[\nu_{lat}]$-learnable via a computable learner.
\end{prop}
\begin{proof}
	For $i\in\omega$, one can easily define a first-order $\exists$-sentence $\psi_i$ which fully describes the finite lattice $\mathcal{D}_i$. We have that: for a structure $\mathcal{S}$, $\mathcal{S}\models \psi_i$ iff the finite lattice $\mathcal{D}_i$ can be isomorphically embedded into $\mathcal{S}$.
	
	Note the following properties of the considered objects:
	\begin{itemize}
		\item $\mathcal{D}_i$ embeds into $\mathcal{B}_j$ if and only if $i=j$.
		
		\item The sequence $\{ \psi_i\}_{i\in\omega}$ is uniformly computable (this follows from the fact that the family  $\{ \mathcal{D}_i\,\colon i\in\omega\}$ is uniformly computable).
		
		\item For every $i$, $\nu_{lat}(i) \cong \mathcal{B}_i$.
	\end{itemize}
	Therefore, one can apply Corollary~\ref{coroll:complexity} with a computable oracle $X$. 
	Proposition~\ref{prop:lattices} is proved.
\end{proof}

Recall that $\mathbb{K}_{L_{\mathrm{lat}}}$ is the class of all countably infinite $L_{\mathrm{lat}}$-structures.

\begin{corollary}
	Suppose that $\nu$ is an arbitrary effective enumeration of the class $\mathbb{K}_{L_{\mathrm{lat}}}$. 
	Then the following holds:
	\begin{itemize}
		\item[(a)] The class $\mathfrak{K}_{lat}$ is $\Inf\Ex_{\cong}[\nu]$-learnable. Note that here the complexity of the learner depends only on the complexity of the sequence $(e_i)_{i\in\omega}$ from Corollary~\ref{coroll:complexity}.
		
	 	\item[(b)] $\mathfrak{K}_{lat}$ is $\Inf\Ex_{\cong}[\nu\oplus\nu_{lat}]$-learnable by a computable learner.
	\end{itemize}	
\end{corollary}

Our second example deals with abelian $p$-groups. We treat abelian groups as structures in the signature $L_{ag}:=\{+,0\}$. 

For a number $i\in\omega$, define the group
\[
	\mathcal{A}_i := \bigoplus_{j\in \omega} \mathbb{Z}(p^{i+1}).
\]
We set $\mathfrak{K}_{ag} := \{ \mathcal{A}_i\,\colon i\in\omega \}$, and we construct a Friedberg effective enumeration $\nu_{ag}$ as follows: just define $\nu_{ag}(i)$ as a natural computable copy of $\mathcal{A}_i$.

\begin{prop}
	The class $\mathfrak{K}_{ag}$ is $\Inf\Ex_{\cong}[\nu_{ag}]$-learnable by a computable learner.
\end{prop}
\begin{proof}
	For $i\in\omega$, one can define a first-order sentence $\psi_i$ which means the following: $\mathcal{S}\models \psi_i$ if and only if $\mathbb{Z}(p^{i+1})$ is a subgroup of $\mathcal{S}$, but $\mathbb{Z}(p^{i+2})$ is not a subgroup of $\mathcal{S}$. Clearly, $\psi_i$ is logically equivalent to a conjunction of an $\forall$-formula (saying that for any element $x$, the condition $p^{i+2} x = 0$ implies $p^{i+1}x=0$) and an $\exists$-formula (saying that there is an element $y$ such that $p^{i+1}y=0$ and $p^{i} y \neq 0$). The rest of the proof is similar to Proposition~\ref{prop:lattices}. Indeed, observe that:
	\begin{itemize}
	\item $\cA_i \models \psi_j$ if and only if $i=j$.
	\item The sequence $\{\psi_i\}_{i\in\omega}$ is uniformly computable.
	\item For every $i$, $\nu_{ag}(i) \cong \mathcal{A}_i$.
	\end{itemize}
	Therefore, one can apply Corollary~\ref{coroll:complexity} to conclude that $\mathfrak{K}_{ag}$ is $\Inf\Ex_{\cong}[\nu_{ag}]$-learnable.
\end{proof}

\begin{corollary}
	Suppose that $\nu$ is an arbitrary effective enumeration of the class $\mathbb{K}_{L_{ag}}$. 
	Then the following holds:
	\begin{itemize}
		\item[(a)] The class $\mathfrak{K}_{ag}$ is $\Inf\Ex_{\cong}[\nu]$-learnable.
		
	 	\item[(b)] $\mathfrak{K}_{ag}$ is $\Inf\Ex_{\cong}[\nu\oplus\nu_{ag}]$-learnable by a computable learner.
	\end{itemize}	
\end{corollary}


\subsection{Boolean algebras}

Proposition~\ref{prop:lattices} provides us with an example of an infinite learnable family of distributive lattices. Here we show that in the realm of Boolean algebras, the situation is dramatically different: informally speaking, one cannot learn even two different isomorphism types of infinite Boolean algebras.

Let $\mathcal{A}$ and $\mathcal{B}$ be structures in the same signature, and $n$ be a non-zero natural number. We write $\mathcal{A} \leq_n \mathcal{B}$ if every infinitary $\Pi_n$ sentence true in $\mathcal{A}$ is also true in $\mathcal{B}$. The relation $\leq_n$ is usually called the \emph{$n$-th back-and-forth relation}.

For a Boolean algebra $\mathcal{C}$, let $\#_{atom}(\mathcal{C})$ denote the cardinality of the set of atoms of $\mathcal{C}$.

\begin{prop}\label{prop:BA}
	Let $\mathfrak{K}$ be some class of infinite Boolean algebras, and let $\nu$ be an effective enumeration of $\mathfrak{K}$. Suppose that $\mathfrak{C}$ is a subclass of $\mathfrak{K}$ such that $\mathfrak{C}$ contains at least two non-isomorphic members. Then the class $\mathfrak{C}$ is not $\Inf\Ex_{\cong}[\nu]$-learnable.
\end{prop}
\begin{proof}
	Suppose that $\mathcal{A}$ and $\mathcal{B}$ are structures from the class $\mathfrak{C}$ such that $\mathcal{A}\not\cong \mathcal{B}$.

	Using the description of the back-and-forth relations on Boolean algebras \cite[\S\,15.3.4]{AK00}, one can prove the following fact: The condition $\mathcal{A} \leq_2 \mathcal{B}$ holds if and only if $\#_{atom}(\mathcal{A}) \geq \#_{atom}(\mathcal{B})$ (see, e.g., Lemma~11 in \cite{Alaev-04} for more details).
	
	 This fact implies that at least one of the following two conditions must be true: 
	 \[
	 	\Sigma_2^{\infi} \text{-} Th(\mathcal{A}) \subseteq \Sigma_2^{\infi} \text{-} Th(\mathcal{B}) \text{ or }
	 	\Sigma_2^{\infi} \text{-} Th(\mathcal{B}) \subseteq \Sigma_2^{\infi} \text{-} Th(\mathcal{A}).
	 \]  
	 Therefore, by Theorem~\ref{thm:Sigma_2-theories}, we deduce that the class $\mathfrak{C}$ is not $\Inf\Ex_{\cong}[\nu]$-le\-ar\-nable.
\end{proof}


\subsection{Linear orders}

First, we show that linear orders exhibit learning properties, which cannot be witnessed by Boolean algebras.

\begin{prop} \label{prop:lo-fin}
	Let $n\geq 2$ be a natural number. Then there is a class of computable infinite linear orders $\mathfrak{C}$ with the following properties:
	\begin{itemize}
		\item[(a)] $\mathfrak{C}$ contains precisely $n$ isomorphism types.
		
		\item[(b)] Suppose that $\mathfrak{K}$ is a superclass of $\mathfrak{C}$, and $\nu$ is an effective enumeration of $\mathfrak{K}$. Then the class $\mathfrak{C}$ is $\Inf\Ex_{\cong}[\nu]$-learnable by a computable learner.
	\end{itemize}
\end{prop}
\begin{proof}
	We show how to build a family $\mathfrak{C}$ containing precisely $k$ non-iso\-mor\-ph\-ic structures. We set
	\[
		\mathfrak{C} = \{ k+\eta+1; (k-1) + \eta + 2; (k-2) + \eta + 3; \ldots ; 2+\eta + (k-1); 1 + \eta + k \}.
	\]
	
	We also define first-order $\exists\forall$-sentences $\psi_i$ as follows: for a linear order $\mathcal{L}$,
	\begin{enumerate}
		\item The sentence $\psi_1$ says that $\mathcal{L}$ has $k$ consecutive elements in the beginning, i.e., there are elements $a_0<a_1<\ldots<a_{k-1}$ such that $a_0$ is the least element and $a_{i+1}$ is the immediate successor of $a_i$, for every $i\leq k-1$.		
		\item For $1<i<k$, $\psi_i$ says that $\mathcal{L}$ has $k-i+1$ consecutive elements in the beginning and $i$ consecutive elements in the end (i.e., there are $b_{i-1} < b_{i-2} < \ldots < b_0$ such that $b_0$ is the greatest and $b_{j+1}$ is the immediate predecessor of $b_j$).
		\item $\psi_k$ says that $\mathcal{L}$ has $k$ consecutive elements in the end.
	\end{enumerate}
	We apply Corollary~\ref{coroll:complexity} to the class $\mathfrak{C}$ and the sequence $\{\psi_i\}_{1\leq i \leq k}$. Thus, we obtain the desired learnability via a computable learner. Proposition~\ref{prop:lo-fin} is proved.
\end{proof}

On the other hand, the next result shows that one still cannot learn \emph{infinite} families of linear orders.

\begin{thm}\label{theo:LO}
	Let $\mathfrak{K}$ be some class of infinite linear orders, and let $\nu$ be an effective enumeration of $\mathfrak{K}$. Suppose that $\mathfrak{C}$ is a subclass of $\mathfrak{K}$ such that $\mathfrak{C}$ contains infinitely many pairwise non-isomorphic members. Then the class $\mathfrak{C}$ is not $\Inf\Ex_{\cong}[\nu]$-learnable.
\end{thm}
\begin{proof}
	The key ingredient of the proof is an analysis of $\Sigma^{\infi}_2$ formulas for linear orders $\mathcal{L}$. First, we define the following auxiliary relations on $\mathcal{L}$:
	\begin{itemize}
		\item A first-order $\forall$-formula $\First(x)$ says that $x$ is the least element of $\mathcal{L}$.
		
		\item An $\forall$-formula $\Last(x)$ says that $x$ is the greatest element of $\mathcal{L}$.
		
		\item An $\forall$-formula $\Succ(x,y)$ says that $x$ and $y$ are consecutive elements, i.e.,
		$(x<y) \ \&\ \neg\exists z (x < z < y)$.
		
		\item A $\Sigma^c_2$ formula $\Block(x,y)$ says the following: either $x=y$, or there are only finitely many elements $z$ between $x$ and $y$ in $\mathcal{L}$. The \emph{block} of an element $x\in\mathcal{L}$ is the set
		\[
			\Block_{\mathcal{L}}[x] := \{ y\,\colon \mathcal{L}\models \Block(x,y)\}.
		\]
	\end{itemize}
	
	\begin{lemma}[{\cite{Mon-10}}] \label{lem:formulas}
		\begin{enumerate}
			\item In the class of countably infinite linear orders, every $\Pi^{\infi}_1$ formula in the signature $\{ \leq\}$ is logically equivalent to a $\Sigma^{\infi}_1$ formula in the signature $\{ \leq, \First, \Last, \Succ\}$.
			
			\item Let $\mathcal{A}$ and $\mathcal{B}$ be countably infinite linear orders. Then we have:
			\[
				\mathcal{A} \leq_2 \mathcal{B}\ \Leftrightarrow\ (\mathcal{A}, \First, \Last, \Succ) \leq_1 (\mathcal{B}, \First, \Last, \Succ).
			\]
		\end{enumerate}
	\end{lemma}
	\begin{proof}
		The proof of~(1) can be recovered from \cite[p.~871]{Mon-10}, see also Lemma II.43 in~\cite{Mon-Book}. 
		
		(2): Recall that the relations $\First$, $\Last$, and $\Succ$ are definable by $\forall$-for\-mu\-las in the signature $\{\leq\}$. This implies that every first-order $\exists$-for\-mu\-la $\psi(\bar x)$ in the signature $\{ \leq, \First, \Last, \Succ\}$ is logically equivalent to a first-order $\exists\forall$-formula $\psi^{[1]}(\bar x)$ in the signature $\{ \leq\}$. 
		
		For a linear order $\mathcal{L}$, let $\mathcal{L}^{\#}$ denote the structure $(\mathcal{L},\First, \Last, \Succ)$. Suppose that $\mathcal{A}^{\#} \nleq_1 \mathcal{B}^{\#}$. Then there is a $\Sigma^{\infi}_1$-sentence 
		\[
			\xi = \underset{i\in I}{\bigvee\skipmm{4}\bigvee} \exists \bar x_i \psi_i (\bar x_i),
		\]
		where $\psi_i$ are quantifier-free, such that $\mathcal{B}^{\#}\models \xi$ and $\mathcal{A}^{\#}\nvDash \xi$. We choose an index $i_0\in I$ such that the $\exists$-sentence $\theta := \exists \bar x_{i_0} \psi_{i_0}(\bar x_{i_0})$ is true in $\mathcal{B}^{\#}$. Clearly, $\mathcal{A}^{\#}\nvDash \theta$. Hence, the $\exists\forall$-sentence $\theta^{[1]}$ is true in $\mathcal{B}$ and false in $\mathcal{A}$. Therefore, $\mathcal{A} \nleq_2 \mathcal{B}$.
		
		Suppose that $\mathcal{A}\nleq_2 \mathcal{B}$. Then there is a $\Sigma^{\infi}_2$-sentence 
		\[
			\xi = \underset{j\in J}{\bigvee\skipmm{4}\bigvee} \exists \bar y_j \psi_j (\bar y_j),
		\]
		where $\psi_j$ are $\Pi^{\infi}_1$-formulas, such that $\mathcal{B}\models \xi$ and $\mathcal{A}\nvDash\xi$. Choose an index $j_0\in J$ such that the formula $\exists \bar y_{j_0} \psi_{j_0}(\bar y_{j_0})$ is true in $\mathcal{B}$. By item~(1), there is a $\Sigma^{\infi}_1$-formula $\lambda(\bar y_{j_0})$ in the signature $\{\leq, \First, \Last, \Succ\}$, which is logically equivalent to $\psi_{j_0}$. In turn, the formula $\exists \bar y_{j_0} \lambda(\bar y_{j_0})$ is logically equivalent to a $\Sigma^{\infi}_1$-sentence $\delta$ in the signature $\{\leq, \First, \Last, \Succ\}$. It is not hard to show that $\mathcal{B}^{\#}\models \delta$ and $\mathcal{A}^{\#}\nvDash \delta$. Therefore, $\mathcal{A}^{\#}\nleq_1\mathcal{B}^{\#}$.	
	\end{proof}
	
	Towards a contradiction, we suppose that there is a family of infinite linear orders $\mathfrak{C} = \{ \mathcal{C}_i \,\colon i\in\omega\}$ such that $\mathfrak{C}$ is $\Inf\Ex_{\cong}[\nu]$-learnable and the structures $\mathcal{C}_i$ are pairwise non-isomorphic. Then by Theorem~\ref{thm:Sigma_2-theories}, there is a sequence of $\Sigma^{\infi}_2$ sentences $(\psi_i)_{i\in\omega}$ such that
	\[
		\mathcal{C}_i \models \psi_j \ \Leftrightarrow\ i=j.
	\]
	We apply Lemma~\ref{lem:formulas}.(1), and for every $i$, we obtain a $\Sigma^{\infi}_1$ sentence $\xi_i$ in the signature $\{ \leq, \First, \Last, \Succ\}$, which is equivalent to $\psi_i$. W.l.o.g., one can choose $\xi_i$ as a finitary $\exists$-sentence: this is because any $\Sigma^{\inf}_1$ sentence $\phi$ is a countable disjuction of finitary $\exists$-formulas, and therefore $\phi$ is true if and only if there is at least one of such $\exists$-formulas which is true. Thus, the intuition behind $\xi_i$ can be explained as follows. The sentence $\xi_i$ describes a finite substructure $\mathcal{F}_i \subset (\mathcal{C}_i, \First, \Last, \Succ)$ such that $\mathcal{F}_i$ cannot be isomorphically embedded into $\mathcal{C}_j$, for $j\neq i$.
	
	Clearly, at least one of the following four cases is satisfied by infinitely many $\mathcal{C}_i$:
	\begin{enumerate}
		\item $\mathcal{C}_i$ has neither least nor greatest elements;
		
		\item $\mathcal{C}_i$ has the least element, but there is no greatest one;
		
		\item $\mathcal{C}_i$ has the greatest element, but there is no least;
		
		\item $\mathcal{C}_i$ has both.
	\end{enumerate}
	Thus, w.l.o.g., one may assume that every $\mathcal{C}_i$ has both least and greatest elements. All other cases can be treated in a way similar to the exposition below.
	
	We give an excerpt from the description \cite[p.~872]{Mon-10} of the relation $\leq_2$ for linear orders.
	
	Let $\mathcal{A}$ be a countably infinite linear order. We define:
	\begin{itemize}
		\item Let $t_0(\mathcal{A}) = n$ if $\mathcal{A} = n + \mathcal{A}_1$, where $n\in\omega$ and the order $\mathcal{A}_1$ has no least element. Set $t_0(\mathcal{A}) = \infty$ if $\mathcal{A} = \omega + \mathcal{A}_1$, where $\mathcal{A}_1$ has no least element.
		
		\item Define $t_2(\mathcal{A}) = m$ if $\mathcal{A} = \mathcal{A}_2 + m$, where $m\in\omega$ and $\mathcal{A}_2$ has no greatest element. Let $t_2(\mathcal{A}) = \infty$ if $\mathcal{A} = \mathcal{A}_2 + \omega^{\ast}$, where $\mathcal{A}_2$ has no greatest element.
	\end{itemize}
	
	As per usual, we assume that $\infty$ is greater than every natural number. We write $\mathcal{A} \equiv_2 \mathcal{B}$ if $\mathcal{A} \leq_2 \mathcal{B}$ and $\mathcal{B} \leq_2 \mathcal{A}$.
	
	\begin{lemma}[{\cite{Mon-10}}] \label{lem:Mon}
		Let $\mathcal{A}$ and $\mathcal{B}$ be countably infinite linear orders. 
		\begin{enumerate}
			\item Suppose that $\max(t_0(\mathcal{A}), t_2(\mathcal{A})) = \infty$. Then, independently of $\mathcal{B}$, we have
			\[
				\mathcal{A} \leq_2 \mathcal{B} \ \Leftrightarrow\ t_0(\mathcal{A}) \geq t_0(\mathcal{B}) \text{ and } t_2(\mathcal{A}) \geq t_2(\mathcal{B}).
			\]
			
			\item Suppose that $\mathcal{A} = n_0 + \mathcal{A}_1 + n_2$ and $\mathcal{B} = m_0 + \mathcal{B}_1 + m_2$, where $n_0,n_2,m_0,m_2\in\omega$, and both $\mathcal{A}_1$ and $\mathcal{B}_1$ have no endpoints. Then 
			\[
				\mathcal{A} \leq_2 \mathcal{B} \ \Leftrightarrow\ (n_0 \geq m_0) \text{ and } (\mathcal{A}_1 \leq_2 \mathcal{B}_1) \text{ and } (n_2 \geq m_2).
			\]
			
			\item Suppose that both $\mathcal{A}$ and $\mathcal{B}$ have no endpoints. Then:
			\begin{itemize}
				\item[(3.1)] If for every non-zero $n\in\omega$, $\mathcal{A}$ has a tuple of $n$ consecutive elements, then $\mathcal{A} \leq_2 \mathcal{B}$.
				
				\item[(3.2)] Suppose that $m$ is a non-zero natural number, and both $\mathcal{A}$ and $\mathcal{B}$ do not have tuples of $m+1$ consecutive elements. If $\mathcal{A}$ has infinitely many tuples of $m$ consecutive elements, then $\mathcal{A} \leq_2 \mathcal{B}$.
			\end{itemize}
		\end{enumerate}
	\end{lemma}	
	
	Lemma~\ref{lem:Mon}.(1) implies the following: if $t_0(\mathcal{A}) = t_0(\mathcal{B}) = \infty$, then we always have either $\mathcal{A} \leq_2 \mathcal{B}$ or $\mathcal{B} \leq_2 \mathcal{A}$. Hence, we deduce that there is at most one structure $\mathcal{C}_i$ with $t_0(\mathcal{C}_i) = \infty$.
	
	A similar argument shows that there is at most one $\mathcal{C}_i$ with $t_2(\mathcal{C}_i) = \infty$. Therefore, w.l.o.g., one can assume that for every $i\in\omega$, both values $t_0(\mathcal{C}_i)$ and $t_2(\mathcal{C}_i)$ are finite. Let
	\[	
		\mathcal{C}_i = m_i + \mathcal{D}_i + n_i,
	\]
	where $m_i, n_i\in\omega$, and the order $\mathcal{D}_i$ has no endpoints. 
	For $i\in\omega$, we define
	\[
		q_i := \sup \{ card(\Block_{\mathcal{D}_i}[x]) \,\colon x \in\mathcal{D}_i\}.
	\]
	
	\begin{claim}\label{claim:aux-001}
		There are only finitely many $i$ with $q_i = \infty$.
	\end{claim}
	\begin{proof}
		For simplicity of exposition, towards a contradiction, suppose that every $q_i$ is infinite. Note that Lemma~\ref{lem:Mon}.(3.1) shows that $\mathcal{D}_i \equiv_2 \mathcal{D}_j$ for all $i$ and $j$.
		
		Since for every $j\neq 0$, we have $\mathcal{C}_j \nleq_2 \mathcal{C}_0$, by Lemma~\ref{lem:Mon}.(2), we obtain that $\mathcal{C}_j$ satisfies at least one of the following two conditions: $m_j < m_0$ or $n_j < n_0$. W.l.o.g., we assume that there are infinitely many $j$ with $m_j < m_0$. Then there is a number $m^{\ast} < m_0$ and an infinite sequence $j[0] < j[1] < j[2] <\dots$ such that $m_{j[k]} = m^{\ast}$ for all $k$. 
		
		Recall that $\mathcal{C}_{j[k]} \nleq_2 \mathcal{C}_{j[0]}$ for all $k \neq 0$. By Lemma~\ref{lem:Mon}.(2), we have $n_{j[k]} < n_{j[0]}$ for every non-zero $k$. Hence, there is a number $n^{\ast} < n_{j[0]}$ such that $n_{j[k]} = n^{\ast}$ for infinitely many $k$. Clearly, if $k\neq k'$ are such numbers, then $\mathcal{C}_{j[k]} \equiv_2 \mathcal{C}_{j[k']}$, which gives a contradiction.
	\end{proof}
	
	By Claim~\ref{claim:aux-001}, one can assume that $q_i<\infty$ for every $i$.
	
	\begin{claim}\label{claim:aux-002}
		There is a number $r\in \omega$ such that $q_i \leq r$ for every $i$.
	\end{claim}
	\begin{proof}
		Again, for simplicity of exposition, assume that $q_0 < q_1 <q_2 <\dots$. Recall that $\mathcal{C}_j\nvDash \xi_0$ for all $j\neq 0$. Suppose that the finite structure $\mathcal{F}_0$ associated with the $\exists$-sentence $\xi_0$ contains precisely $l_0$ elements. 
		
		Choose $j^{\ast}$ such that $q_{j^{\ast}} \geq 2 l_0$. Clearly, for every $j \geq j^{\ast}$, the order $\mathcal{D}_j$ contains at least one block of size at least $2 l_0$. Thus, $\mathcal{F}_0$ cannot be embedded into $\mathcal{C}_j$ only because of one of the following two obstacles:
		\begin{itemize}
			\item $m_j < m_0$, i.e., the size of the first (under $\leq_{\mathcal{C}_j}$) block in $\mathcal{C}_j$ is too small for an appropriate embedding; or
			
			\item $n_j < n_0$, i.e., the size of the last block in $\mathcal{C}_j$ is too small.
		\end{itemize}
		The relation $\Succ^{\mathcal{C}_j}$ won't give us any problems, since one can embed all the $\mathcal{F}_0$-blocks (except the first one and the last one) inside a $\mathcal{D}_j$-block of size $\geq 2 l_0$.
		
		As in Claim~\ref{claim:aux-001}, we can assume that there is a number $m^{\ast} < m_0$ such that $m_j = m^{\ast}$ for infinitely many $j\geq j^{\ast}$. Form an increasing sequence $j[0] < j[1] < j[2] < \dots$ of these $j$. Recall that $q_{j[l]} < q_{j[l+1]}$ for all $l\in\omega$. Re-iterating the argument above, we obtain that there is a number $n^{\ast} < n_{j[0]}$ such that there are infinitely many $l$ with $n_{j[l]} = n^{\ast}$. Choose a sequence $l[0] < l[1] < l[2] < \dots$ of these $l$. Suppose that the structure $\mathcal{F}_{j[l[0]]}$ contains precisely $t_1$ elements.
		
		Find the least $l^{\ast} = l[s^{\ast}]$ with $q_{j[l^{\ast}]} \geq 2 t_1$. Recall that we have $m_{j[l^{\ast}]} = m_{j[l[0]]} = m^{\ast}$ and $n_{j[l^{\ast}]} = n_{j[l[0]]} = n^{\ast}$. Thus, as before, it is not hard to show that the structure $\mathcal{F}_{j[l[0]]}$ can be embedded into $\mathcal{C}_{j[l^{\ast}]}$. This shows that $\mathcal{C}_{j[l^{\ast}]} \models \xi_{j[l[0]]}$, which gives a contradiction.
	\end{proof}
	
	By Claim~\ref{claim:aux-002}, we obtain that
	\[
		r := \sup \{ q_i\,\colon i\in\omega\} < \infty.
	\]
	Moreover, we will assume that $q_i = r$ for all $i\in\omega$: indeed,
	\begin{itemize}
		\item If there are only finitely many $i$ with $q_i = r$, then we just delete the corresponding structures $\mathcal{C}_i$. After that the value $r$ goes down.
		
		\item If there are already infinitely many $i$ with $q_i = r$, then we delete all $\mathcal{C}_j$ with $q_j < r$.
	\end{itemize}
	
	\begin{claim}\label{claim:aux-003}
		There are only finitely many $i$ such that the order $\mathcal{D}_i$ has infinitely many blocks of size $r$.
	\end{claim}
	\begin{proof}
		Again, for simplicity, assume that every $\mathcal{D}_i$ has infinitely many blocks of size $r$. Since $q_i = r$ for all $i$, Lemma~\ref{lem:Mon}.(3.2) implies that $\mathcal{D}_i \equiv_2 \mathcal{D}_j$ for all $i$ and $j$.
		
		As in Claim~\ref{claim:aux-002}, $\mathcal{F}_0$ is not embeddable into $\mathcal{C}_j$, $j\neq 0$, and this is witnessed by one of the following: either $m_j < m_0$ or $n_j < n_0$. We recover a number $m^{\ast} < m_0$ and a sequence $j[0] < j[1] < j[2] < \dots$ such that $m_{j[l]}=m^{\ast}$ for all $l$.
		
		The finite structure $\mathcal{F}_{j[0]}$ is not embeddable into $\mathcal{C}_{j[l]}$, $l\neq 0$. By Lemma \ref{lem:Mon}.(2), this implies that $n_{j[l]} < n_{j[0]}$ for non-zero $l$. Again, there is a number $n^{\ast} < n_{j[0]}$ and a sequence $l[0] < l[1] < l[2] <\dots$ such that $n_{j[l[s]]} = n^{\ast}$ for all $s$. This shows that $\mathcal{C}_{j[l[1]]} \models \xi_{j[l[0]]}$, and this yields a contradiction.
	\end{proof}
	
	Claim~\ref{claim:aux-003} implies that one may assume the following: each $\mathcal{D}_i$ has only finitely many blocks of size $r=q_i$.
	
	The rest of the proof is only sketched, since all the key ideas are already present. Let $\#(r;i)$ denote the number of blocks of size $r$ inside $\mathcal{D}_i$.
	
	\begin{claim} \label{claim:aux-004}
		There is a number $N$ such that $\#(r;i) \leq N$ for all $i$.
	\end{claim}
	\begin{proof}
		Assume that $\#(r;i) < \#(r;i+1)$ for all $i$. As before, the finite structure $\mathcal{F}_0$ cannot be embedded into $\mathcal{C}_j$, where $j$ is large enough, and this can be witnessed only by one of the following conditions: $m_j < m_0$ or $n_j < n_0$ for such $j$. Hence, we assume that there is a sequence $j[0] < j[1] < j[2] < \dots$ with $m_{j[l]} = m^{\ast} < m_0$ for all $l$. By considering possible embeddings of the finite structure $\mathcal{F}_{j[0]}$, we recover a sequence $l[0] < l[1] < l[2] < \dots$  with $n_{j[l[s]]} = n^{\ast} < n_{j[0]}$ for all $l$. Clearly, $\mathcal{F}_{j[l[0]]}$ can be embedded into any $\mathcal{C}_{j[l[s]]}$, where $s$ is large enough, and this produces a contradiction.
	\end{proof}
	
	By Claim~\ref{claim:aux-004}, one can assume that $\#(r;i) = N <\infty$ for all $i$. For simplicity, consider $N = 2$. Then every $\mathcal{D}_i$ can be presented in the following form:
	\[
		\mathcal{D}_i = \mathcal{D}_{i,0} + r + \mathcal{D}_{i,1} + r + \mathcal{D}_{i,2}, \text{ where}
	\]
	\begin{itemize}
		\item every $\mathcal{D}_{i,j}$ does not have endpoints, and
		
		\item every block inside $\mathcal{D}_{i,j}$ has size at most $r-1$.
	\end{itemize}
	After that, one needs to write a cumbersome proof by recursion in $r$. The arrangement of this recursion can be recovered from the ideas from~\cite[p.~872]{Mon-10}. 
	
	In our case, the first stage of recursion will roughly consist of the following claims:
	\begin{itemize}	
		\item[(a)] We say that a block of size $(r-1)$ is \emph{large}. Then one can prove that there are only finitely many $i$ such that every $\mathcal{D}_{i,j}$ contains infinitely many large blocks.
		
		\item[(b)] If there are infinitely many $i$ such that, say, both $\mathcal{D}_{i,0}$ and $\mathcal{D}_{i,1}$ contain infinitely many large blocks, then one can assume that there is a number $N_1$ such that every $\mathcal{D}_{i,2}$ has at most $N_1$ large blocks. In this case, the next stage of recursion will play essentially only with $\mathcal{D}_{i,2}$.
		
		\item[(c)] Assume that there are infinitely many $i$ such that $\mathcal{D}_{i,0}$ has infinitely many large blocks, but every $\mathcal{D}_{i,1}$ and $\mathcal{D}_{i,2}$ has only finitely many  large blocks. Then there are three main variants:
		\begin{itemize}
			\item[(c.1)] There are a number $N_2$ and a sequence $i_0 < i_1 < i_2 < \dots$ such that for every $k$, $\mathcal{D}_{i_k,1}$ has precisely $N_2$ large blocks and $\mathcal{D}_{i_k,2}$ contains, say, at least $k$ large blocks. Then one needs to invoke recursion for $\mathcal{D}_{i,1}$.
			
			\item[(c.2)] A case similar to the previous one, but here we require that every $\mathcal{D}_{i_k,1}$ has at least $k$ large blocks. Then one can obtain a contradiction.
			
			\item[(c.3)] There is a number $N_3$ such that every $\mathcal{D}_{i,1}$ or $\mathcal{D}_{i,2}$ has at most $N_3$ large blocks. Then proceed to the next recursion stage by considering both $\mathcal{D}_{i,1}$ and $\mathcal{D}_{i,2}$ simultaneously.
		\end{itemize}
		
		\item[(d)] Assume that each $\mathcal{D}_{i,j}$ has only finitely many large blocks. The main cases are as follows:
		\begin{itemize}
			\item[(d.1)] There are a number $N_4$ and a sequence $i_0 < i_1 < i_2 < \dots$ such that for every $k$, $\mathcal{D}_{i_k,0}$ contains precisely $N_4$ large blocks and each of $\mathcal{D}_{i_k,1}$ and $\mathcal{D}_{i_k,2}$ has at least $k$ large blocks. Then one calls recursion for $\mathcal{D}_{i,0}$.
			
			\item[(d.2)] A case similar to the previous one, but now we require that $\mathcal{D}_{i_k,1}$ always keeps precisely $N_5$ large blocks. Then the next recursion stage will work with $\mathcal{D}_{i,0}$ and $\mathcal{D}_{i,1}$ simultaneously.
			
			\item[(d.3)] For all $k$, every block $\mathcal{D}_{i_k, j}$ contains at least $k$ large blocks. This leads to a contradiction.
			
			\item[(d.4)] There is a number $N_6$ such that each $\mathcal{D}_{i,j}$ contains at most $N_{6}$ large blocks. Then we go to the next stage of recursion, and we have to consider all $\mathcal{D}_{i,j}$ simultaneously.
		\end{itemize}
	\end{itemize}
	
	When the outlined recursion procedure finishes, we will get a contradiction in all considered cases. This implies that the class $\mathfrak{C}$ cannot be $\Inf\Ex_{\cong}[\nu]$-le\-ar\-nable. Theorem~\ref{theo:LO} is proved.
\end{proof}

\section{Conclusions and open problems}
In this paper, we investigated the problem of learning computable structures up to isomorphism. We used infinitary logic to offer a model-theoretic characterization of which families of structures are $\Inf\Ex$-learnable. 
Applying such a characterization,  we proved  that our learning paradigm is very sensitive to the algebraic properties of the structures to be learned: e.g., while there is
an infinite learnable family of distributive lattices, no infinite family of linear orders is learnable.

Many questions remain open. In particular, one shall ask which families of structures can be learned when only positive data of the target structure is available. The ideal goal would be to obtain an analogue of Theorem \ref{thm:Sigma_2-theories} for the learning type $\Text\Ex_{\cong}$ (already introduced in \cite{FKS-ta}). Moreover, one obtains natural variants of the learning problems considered in this paper by replacing isomorphism with weaker notions (such as the bi-embeddability relation discussed in \cite{FKS-ta}) or with stronger ones (such as computable isomorphism).

Finally, in this paper we only marginally considered the complexity of the learners described. We still have a limited understanding of which families of structures can be learned by a learner of a given fixed complexity. In this direction, the following question looks particularly intriguing: is there a pair of two (non-isomorphic) structures which is $\Inf\Ex_{\cong}$-learnable, but no computable learner can learn it?

\end{document}